\pdfoutput=1
\documentclass[]{article}
\usepackage{amssymb,amsmath,hyperref,verbatim,graphicx,amsthm}
\theoremstyle{remark}

\makeatletter \let\cl@chapter\relax \makeatother
\usepackage{cleveref, xurl}
\usepackage[ruled,vlined,linesnumbered]{algorithm2e}
\usepackage{multirow, longtable}
\usepackage{booktabs}
\usepackage{diagbox}
\usepackage{authblk}
\usepackage{enumitem}
\theoremstyle{plain}
\newtheorem{theorem}{Theorem}[section]
\newtheorem{proposition}[theorem]{Proposition}
\newtheorem{corollary}[theorem]{Corollary}
\newtheorem{lemma}[theorem]{Lemma}
\usepackage{pdflscape}
\usepackage{mathtools}

\usepackage[dvipsnames]{xcolor}
\usepackage{tikz}
\usepackage{pgfplots}
\pgfplotsset{width=10cm,compat=1.9}
\usepackage{amsmath, amssymb}
\usepackage{booktabs}
\usepackage[ruled,vlined,linesnumbered]{algorithm2e}
\usepackage{hyperref}
\usepackage{subcaption}
\hypersetup{
    colorlinks=true,
    linkcolor=blue,
    filecolor=magenta,      
    urlcolor=cyan,
}
\usepackage{natbib}
\usepackage{threeparttable}
\usepgfplotslibrary{fillbetween}
\usepackage{comment}
\usetikzlibrary{calc,intersections}
\usetikzlibrary{angles,quotes}

\begin{document}

\newcommand{\red}{\color{red}}
\newcommand{\blue}{\color{blue}}
\newcommand{\green}{\color{green}}
\allowdisplaybreaks

\title{Polyhedral Outer Approximation for MISOCP: Geometry and Progressive Integrality}
\author[]{Yongzheng Dai}
\affil[]{ISyE, Georgia Institute of Technology, Atlanta, GA, USA}
\date{}
\maketitle

\begin{abstract}
Polyhedral outer approximation (OA) is widely used to solve mixed-integer second-order cone programs (MISOCPs) by iteratively strengthening polyhedral relaxations of conic constraints. However, the accumulation of cutting planes can increase the difficulty of solving the resulting mixed-integer linear programs (MILPs), while repeatedly solving these MILPs can be computationally expensive. To address these challenges, we investigate the geometry of supporting cuts and the computational role of integrality in OA. First, we establish theoretical bounds showing that the marginal contribution of a new supporting cut decreases at least linearly as its supporting direction approaches an existing one. Motivated by this analysis, we develop a closed-form cut-generation strategy that balances separation depth and directional diversity. Second, we propose a progressive-integrality OA framework that strengthens the polyhedral approximation through linear programs and partially integral relaxations before solving full MILPs. We theoretically characterize the transfer of cut-induced bound improvements across integrality levels and provide a branch-and-bound interpretation of the framework. Computational experiments on CBLIB instances and large-scale AC unit-commitment models demonstrate substantial improvements over classical OA and direct MISOCP solution, with runtime reductions of up to 62\% relative to classical OA on the reported power-system instances.

\end{abstract}
\section{Introduction}
Consider a mixed-integer second-order cone program (MISOCP):
\begin{equation}\label{eq:misocp}
    \min_{x\in \mathbb R^n}\{c^\top x: x\in \mathcal C \cap \mathcal X\},
\end{equation}
where $\mathcal C$ is an intersection of second-order cones (SOCs) and
\[\mathcal X :=\{x\in \mathbb R^n: Ax \leq b, x_j\in \mathbb Z, \forall j\in \mathcal J\}\]
is a bounded mixed-integer polyhedral set. We do not assume any additional structure on Problem~\eqref{eq:misocp}, such as a decomposition or a specific formulation.

MISOCPs arise in a broad range of applications, including assortment optimization \citep{SAK:assortment}, portfolio and risk optimization \citep{ghaoui2003worst}, network design and operations \citep{byeon2019unit}, and statistical learning \citep{kucukyavuz2023consistent}; see \citep{benson2013mixed} for a survey. They also provide useful convex formulations and relaxations for mixed-integer nonlinear problems, with important applications in power systems \citep{bai2015decomposition,dai2026solving,dai2026scheduling}. Although continuous SOCPs can be solved efficiently, the presence of discrete variables makes MISOCPs computationally challenging \citep{bonami2011algorithms}. Representative methods include lifted polyhedral approximations of second-order cones \citep{ben2001polyhedral}, branch-and-bound for mixed-integer conic quadratic programs \citep{vielma2008lifted}, conic mixed-integer cuts \citep{atamturk2010conic}, and more recent polyhedral and conic outer-approximation frameworks (OAs) \citep{lubin2018polyhedral,coey2020outer}. Also, decomposition methods for structure-specific problems, e.g., Lagrangian relaxation (LR) methods \citep{muckstadt1977application,zhuang2002towards,tuncer2023an} and Benders decomposition \citep{nasri2015network,constante2024security}.

In this work, we study Problem~\eqref{eq:misocp} through a polyhedral OA framework, which replaces the conic region $\mathcal C$ with a polyhedral relaxation $\mathcal P\supseteq\mathcal C$ and iteratively strengthens it by valid cuts. Polyhedral OA has a long history in convex optimization and convex mixed-integer nonlinear programming (MINLP), tracing back to cutting-plane methods of \cite{kelley1960cutting} and \cite{veinott1967supporting}, and to classical OA and linear/nonlinear-programming-based methods for convex MINLP \citep{duran1986outer,quesada1992lp,fletcher1994solving}. Supporting-hyperplane variants were further developed by \cite{kronqvist2016extended} and geometrically related to Kelley’s cutting-plane method \cite{serrano2020relation}.

Despite its generality, OA faces two computational challenges. First, strengthening the relaxation typically requires accumulating supporting cuts, which increases the size and the risk of numerical difficulties in the resulting MILPs. Although cut selection and management are well-established components of MIP algorithms \citep{abhishek2010filmint,turner2023adaptive}, existing strategies are largely based on measures of individual cut quality or heuristic filtering (locally), while the effectiveness of a new cut should also depend on its contribution relative to added cuts (globally). Recent work has further emphasized that the strength of an inequality should be assessed with respect to the polyhedron that it strengthens \citep{warme2026quantitativeI,warme2026quantitativeII}. This raises a fundamental question: \textbf{(Q1) how can the marginal contribution of a new supporting cut be quantified and used to guide cut generation?} 

Second, conventional OA repeatedly solves full MILP relaxations, potentially requiring substantial computational effort before the conic approximation becomes sufficiently accurate. Previous approaches have explored using LP relaxations during early OA iterations \cite{kronqvist2016extended,muts2020decomposition}. However, this also raises the question: \textbf{(Q2) how improvements obtained at weaker integrality levels contribute to subsequent full-MILP relaxations and how integrality should be progressively incorporated into the approximation process?}

We address these questions through geometric analysis of SOC supporting cuts and a progressive-integrality OA framework. Our main contributions are summarized as follows.
\begin{enumerate}
    \item We study the marginal contribution of a new SOC supporting cut relative to those already added in the relaxation. Using both volume-based and violation-based measures, we establish theoretical bounds showing that this contribution decreases at least linearly as the new supporting direction approaches an existing one. Motivated by these geometric results, we formulate cut generation as a tradeoff between separation depth and directional diversity, derive closed-form characterizations of the resulting supporting cuts, and develop an adaptive cut-generation strategy.
    \item We develop a progressive-integrality OA framework that gradually introduces integers, proceeding from LP relaxations through partially integral relaxations before solving full MILPs. We theoretically analyze how cut-induced bound improvements at intermediate integrality levels transfer to the full-MILP dual bound and establish a branch-and-bound interpretation for how cuts generated at different integrality levels can strengthen the root relaxation and upper levels of subsequent search trees.
    \item Experiments on MISOCP instances from CBLIB \citep{friberg2016cblib} and large-scale AC unit-commitment models demonstrate the benefits of both the geometry-aware cut-generation strategy and progressive integrality.
\end{enumerate}

The remainder of the paper is organized as follows. Section~\ref{sec:cutting_plane} develops the geometric analysis of SOC supporting cuts and the depth-efficiency balanced cut-generation strategy. Section~\ref{sec:progressive-oa} introduces the progressive-integrality OA framework and the associated primal-recovery procedure. Section~\ref{sec:experiement} presents the computational experiments. Proofs, algorithmic details, implementation settings, and extended computational results are provided in the appendices.

\section{Cutting Planes for Second-Order Cone}\label{sec:cutting_plane}

This section studies supporting cuts for polyhedral OA of the second-order cone, with particular emphasis on the marginal contribution of a new cut relative to those already present in the relaxation. Throughout this section, we assume all variables are bounded. Since the analysis in this section concerns a single second-order cone, we write SOC as
\[L^N := \{(x, s)\in \mathbb{R}^{N-1}\times \mathbb{R}_+: \|x\|_2 \leq s\}.\]

\subsection{Geometry of the Deepest SOC Cut}
Let $(x^{(1)}, s^{(1)}) \not\in L^N$. Among all valid linear inequalities for  $L^N$, the deepest cut at $(x^{(1)}, s^{(1)})$ is obtained from the separation problem
\begin{equation}\label{eq:deep_cut_problem}
    \begin{aligned}
        \max_{a,b}\ &a^\top (x^{(1)}, s^{(1)}) - b\\
        \mbox{s.t.}\ &a^\top (x, s) \leq b,\ \forall (x,s)\in L^N,\\
        &\|a\|_2 \leq 1.
    \end{aligned}
\end{equation}
Building on Proposition 5 by \cite{bienstock2026accurate}, the following result gives an equivalent geometric characterization of the deepest cut.
\begin{proposition}\label{prop:deepest_cut}
    Consider $(x^{(1)}, s^{(1)}) \not\in L^N$ with $s^{(1)} \geq 0$. Let $(y^{(1)}, t^{(1)}) := \mathrm{Proj}_{L^N}(x^{(1)}, s^{(1)})$. Then $\frac{y^{(1)}}{\|y^{(1)}\|_2} = \frac{x^{(1)}}{\|x^{(1)}\|_2}$, and the inequality 
    \[(y^{(1)}/\|y^{(1)}\|_2)^\top x\leq s\] 
    is (a) the deepest cut solving Problem~\eqref{eq:deep_cut_problem} after scaling by $1/\sqrt{2}$; (b) a supporting inequality of $L^N$; (c) the gradient cut generated at the projection point $(y^{(1)}, t^{(1)})$.
\end{proposition}

Proposition~\ref{prop:deepest_cut} shows that the conventional deepest cut admits three equivalent interpretations: it maximizes the separation of the current infeasible point, it supports the cone at its Euclidean projection, and it coincides with the first-order cut generated at that projection. In particular, its supporting direction is aligned with $x^{(1)}$. OA based on such cuts converges to a global optimum or to a prescribed accuracy in finite time; see, e.g., \citep{veinott1967supporting,kronqvist2016extended}.

We next analyze what this cut implies for subsequent relaxation solutions. After adding the cut based on $(x^{(1)}, s^{(1)})$, define the strengthened polyhedron as
\[P := \{(x,s)\in \mathbb{R}^{N-1}\times \mathbb{R}_+: (y^{(1)}/\|y^{(1)}\|_2)^\top x\leq s\} \supseteq L^N.\]

\begin{theorem}\label{thm:dist_bound}
    Let $(x^{(2)}, s^{(2)}) \in P$. Then we have
    \begin{equation}\label{eq:euclid_bound}
        \mathrm{dist}((x^{(2)}, s^{(2)}), L^N)\leq \frac{\sqrt{2}\|x^{(2)} - y^{(1)}\|_2}{2}.
    \end{equation}
    If $x^{(2)} \neq 0$, let $\theta\in[0,\pi]$ denote the angle between $y^{(1)}$ and $x^{(2)}$. Then
    \begin{equation}\label{eq:angle_bound}
        \mathrm{dist}((x^{(2)}, s^{(2)}), L^N)\leq \frac{\sqrt{2}(1-\cos(\theta))\|x^{(2)}\|_2}{2}.
    \end{equation}
\end{theorem}

From \cite[Proposition 5]{bienstock2026accurate}, the violation of $(x^{(2)}, s^{(2)})\not\in L^N$ to the constraint $\|x\|_2 \leq s$ is
\begin{equation}\label{eq:vio_dist}
    \|x^{(2)}\|_2 - s^{(2)} = \sqrt{2}\cdot\mathrm{dist}((x^{(2)}, s^{(2)}), L^N).
\end{equation}
Therefore, Theorem~\ref{thm:dist_bound} controls not only the Euclidean distance to $L^N$, but also the violation of the SOC constraint. In particular, Equation \eqref{eq:euclid_bound} shows that the residual violation is small when $x^{(2)}$ is close to the previous supporting point $y^{(1)}$, while Equation~\eqref{eq:angle_bound} provides the corresponding directional interpretation: for a fixed $\|x^{(2)}\|_2$, the residual violation vanishes as the angle $\theta$ approaches zero.

This observation reveals a limitation of evaluating a cut solely by its separation depth at the current infeasible point. If the corresponding supporting direction is already well represented by an existing cut, the current polyhedral approximation may already be accurate in that direction, and an additional nearby cut may provide limited new information. This motivates studying the marginal contribution of a new supporting cut relative to existing cuts.

\subsection{Marginal Efficiency of Supporting Cuts}

Suppose that $m$ supporting cuts, associated with $y^{(1)}, \ldots, y^{(m)}$, have already been added to the relaxation:
\[P(m) :=\{(x,s)\in \mathbb{R}^{N-1}\times \mathbb{R}_+: (y^{(i)}/\|y^{(i)}\|_2)^\top x\leq s, \forall i=1,...,m\} \supseteq L^N.\]

We study the marginal contribution of an additional supporting cut generated from $y^{(m+1)}$ relative to the added cuts $y^{(1)}, \ldots, y^{(m)}$ from two complementary perspectives: reduction of the volume of the polyhedral relaxation and reduction of an upper bound on the residual SOC violation

\subsubsection{Volume-based measure.} We begin with a 3D example. Suppose $N = 3$, $m = 2$, and $s \leq U_r \in \mathbb{R}_+$. In Figure~\ref{fig:example}, we present a slice of the second-order cone $L^3$ in $s = r > 0$. We have two cuts (Cut 1 and Cut 2) that support the circle at points $a,b$ respectively, and they intersect at point $d$. Let the angle $\angle aob = \theta_1 + \theta_2 < \pi$. The area of the sector of the outer approximation bounded by the two supporting cuts (area $aobd$) is $r^2\tan{\frac{\theta_1 + \theta_2}{2}}$; thus, the volume of $P$ is
\[\mathrm{Vol}(P(2)) := \int_0^{U_r} r^2\tan{\frac{\theta_1 + \theta_2}{2}}\, dr = \frac{U_r^3}{3}\tan{\frac{\theta_1 + \theta_2}{2}}.\]
Now we add a new cut (Cut 3) at point $c$ between $a$ and $b$, and the new cut intersects Cut 1 and Cut 2 at points $e$ and $f$ separately. Let angles $\angle aoc = \theta_1$ and $\angle boc = \theta_2$. Then Cut 3 removes triangle $efd$ from each slice, whose area is  $r^2\tan{\frac{\theta_1}{2}}\tan{\frac{\theta_2}{2}}\tan{\frac{\theta_1 + \theta_2}{2}}$. Hence, the volume reduction from $P(2)$ to $P(3)$ is
\begin{equation*}
    \begin{aligned}
        \mathrm{Vol}(P(2)) - \mathrm{Vol}(P(3)) &= \int_0^{U_r} r^2\tan{\frac{\theta_1}{2}}\tan{\frac{\theta_2}{2}}\tan{\frac{\theta_1 + \theta_2}{2}}\, dr\\
        &= \frac{U_r^3}{3}\tan{\frac{\theta_1}{2}}\tan{\frac{\theta_2}{2}}\tan{\frac{\theta_1 + \theta_2}{2}}.
    \end{aligned}
\end{equation*}

\begin{figure}[!htbp]
    \centering
\begin{subfigure}[b]{0.48\textwidth}
    \centering
\begin{tikzpicture}[
    scale=3.5,
    line cap=round,
    line join=round,
    point/.style={circle,fill=black,inner sep=1.2pt},
    tangent/.style={thick},
    auxiliary/.style={dashed,thin}
]
\coordinate (O) at (0,0);
\draw[->] (0,0) -- (1.35,0) node[right] {$x_1$};
\draw[->] (0,0) -- (0,1.35) node[above] {$x_2$};

\draw (1,0.015) -- (1,-0.015) node[below=2pt] {$r$};
\draw (0.015,1) -- (-0.015,1) node[left=2pt] {$r$};

\node[below left] at (0,0) {$O$};

\draw[thick] (1,0)
    arc[start angle=0,end angle=90,radius=1];

\coordinate (A) at ({cos(15)},{sin(15)});
\coordinate (C) at ({cos(30)},{sin(30)});
\coordinate (B) at ({cos(75)},{sin(75)});

\coordinate (D) at
    ({cos(45)/cos(30)},
     {sin(45)/cos(30)});

\coordinate (E) at
    ({cos(22.5)/cos(7.5)},
     {sin(22.5)/cos(7.5)});

\coordinate (F) at
    ({cos(52.5)/cos(22.5)},
     {sin(52.5)/cos(22.5)});

\draw[auxiliary] (0,0) -- (A);
\draw[auxiliary] (0,0) -- (B);
\draw[auxiliary] (0,0) -- (C);


\draw[thick] (E) -- (D);
\draw[thick] (F) -- (D);
\draw[thick] (E) -- (F);

\node[point,label={[below right=1pt]$a$}] at (A) {};
\node[point,label={[right=2pt]$c$}] at (C) {};
\node[point,label={[above left=1pt]$b$}] at (B) {};

\node[point,label={[above right =1.5pt]$d$}] at (D) {};
\node[point,label={[right=2pt]$e$}] at (E) {};
\node[point,label={[above right=-1.5pt]$f$}] at (F) {};
\node[label = {[above, blue] Cut $2$}] at (1.2, 0.72) {};
\node[label = {[above, blue] Cut $1$}] at (0.7, 1.15) {};
\node[label = {[above, red] Cut $3$}] at (0.5, 1.05) {};

\coordinate (Astart) at
    ({cos(15)+0.15*sin(15)},
     {sin(15)-0.15*cos(15)});

\coordinate (Aend) at
    ({cos(15)-0.90*sin(15)},
     {sin(15)+0.90*cos(15)});
\coordinate (Bstart) at
    ({cos(75)+0.90*sin(75)},
     {sin(75)-0.90*cos(75)});

\coordinate (Bend) at
    ({cos(75)-0.15*sin(75)},
     {sin(75)+0.15*cos(75)});
\coordinate (Cstart) at
    ({cos(30)+0.25*sin(30)},
     {sin(30)-0.25*cos(30)});

\coordinate (Cend) at
    ({cos(30)-0.65*sin(30)},
     {sin(30)+0.65*cos(30)});

\path[name path=tangentA] (Astart) -- (Aend);
\path[name path=tangentB] (Bstart) -- (Bend);
\path[name path=tangentC] (Cstart) -- (Cend);
\draw[tangent, blue] (Astart) -- (Aend);
\draw[tangent, blue] (Bstart) -- (Bend);
\draw[tangent, red] (Cstart) -- (Cend);
\pic[
    red,
    thick,
    draw,
    "$\theta_1$",
    angle radius=7mm,
    angle eccentricity=1.5
] {angle=A--O--C};

\pic[
    red,
    thick,
    draw,
    "$\theta_2$",
    angle radius=5mm,
    angle eccentricity=1.45
] {angle=C--O--B};
\end{tikzpicture}
\caption{A slice of $s = r$, $N = 3$.}
\label{fig:example}
\end{subfigure}
\hfill
\begin{subfigure}[b]{0.48\textwidth}
    \centering
    \begin{tikzpicture}[
    scale=3.5,
    line cap=round,
    line join=round,
    point/.style={circle,fill=black,inner sep=1.2pt},
    tangent/.style={thick},
    auxiliary/.style={dashed,thin}
    ]

    \def\R{1}
    \def\angleone{80}
    \def\angletwo{70}
    \def\anglethree{15}
    \def\anglefour{42.5}
    \pgfmathsetmacro{\Rtwo}
        {\R/cos(\angletwo-\anglefour)}
    \draw[->] (0,0) -- (1.35,0) node[right] {$x_1$};
    \draw[->] (0,0) -- (0,1.35) node[above] {$x_2$};

    \draw[thick] (\R,0)
        arc[start angle=0,end angle=90,radius=\R];

    \fill (0,0) circle (0.015);
    \node[below left] at (0,0) {$O$};

    \coordinate (P1) at
        ({\R*cos(\angleone)},{\R*sin(\angleone)});
    \coordinate (P2) at
        ({\R*cos(\angletwo)},{\R*sin(\angletwo)});
    \coordinate (P3) at
        ({\R*cos(\anglethree)},{\R*sin(\anglethree)});
    \coordinate (P4) at
        ({\Rtwo*cos(\anglefour)}, {\Rtwo*sin(\anglefour)});
    \coordinate (P4c) at
        ({\R*cos(\anglefour)}, {\R*sin(\anglefour)});

    \draw[thick, dashed] (0,0) -- (P1);
    \draw[thick, dashed] (0,0) -- (P2);
    \draw[thick, dashed] (0,0) -- (P3);
    \draw[thick, dashed] (0,0) -- (P4);

    \fill[cyan] (P1) circle (0.01);
    \node[cyan, above] at (P1) {$y^{(1)}$};

    \fill[cyan] (P2) circle (0.01);
    \node[cyan, above right] at (P2) {$y^{(2)}$};

    \fill[cyan] (P3) circle (0.01);
    \node[cyan, above right] at (P3) {$y^{(3)}$};

    \draw[cyan, thick]
        ($(P1)-1*({-sin(\angleone)},{cos(\angleone)})$)
        --
        ($(P1)+0.25*({-sin(\angleone)},{cos(\angleone)})$);

    \draw[cyan, thick]
        ($(P2)-1.2*({-sin(\angletwo)},{cos(\angletwo)})$)
        --
        ($(P2)+0.35*({-sin(\angletwo)},{cos(\angletwo)})$);

    \draw[cyan, thick]
        ($(P3)-0.35*({-sin(\anglethree)},{cos(\anglethree)})$)
        --
        ($(P3)+0.85*({-sin(\anglethree)},{cos(\anglethree)})$);

    \fill[red] (P4) circle (0.01);
    \node[red, above right] at (P4) {$\bar x$};
    \draw[red, thick] (P4) -- (P4c);
    \node[red, right] at (P4c) {$\Delta$};

\end{tikzpicture}
    \caption{A slice of $s = 1$, $N = 3$}
    \label{fig:example2}
\end{subfigure}
\caption{Two examples for volume-based and violation-based measures.}
\end{figure}

If either $\theta_1 \to 0$ or $\theta_2 \to 0$, the volume removed by the new cut converges to zero linearly in the corresponding angle. Thus, even in this elementary setting, a cut that approaches an existing supporting direction has vanishing marginal geometric effect. The same phenomenon holds in any arbitrary dimension.

\begin{theorem}\label{thm:vol_reduction}
    Suppose we add a new supporting cut based on $y^{(m+1)}$ to $P(m)$ and get $P(m+1)$. Let $\theta_i \in [0,\pi]$ be the angle between $y^{(m+1)}$ and $y^{(i)}$, and $\bar \theta := \min_{i} \theta_i$. If $|x_j| \leq U \in \mathbb R_+$ for all $j=1,...,N-1$, and $s \leq U_r \in \mathbb R_+$, then
    \begin{equation}\label{eq:vol_reduction}
        \mathrm{Vol}(P(m)) - \mathrm{Vol}(P(m+1)) = O(\bar\theta).
    \end{equation}
\end{theorem}

We note that the RHS of Equation~\eqref{eq:vol_reduction} can be sharpened to $\Theta(\bar \theta)$ with some additional assumptions; see Appendix~\ref{app:extend_vol_reduction}. If $\bar \theta$ is small, i.e., the direction of the supporting point for the new cut is too close to some existing supporting points, the volume reduction disappears; in other words, the contribution of the new cut is limited in terms of volume reduction.

\subsubsection{Violation-based measure.} Consider the maximal violation of $P(m)$
\[\Delta(P(m)) := \max_{(x,s)\in P(m)}([\|x\|_2 - s]_+),\]
where $[\cdot]_+$ is a projection to $\mathbb R_+$. If $\Delta(P(m)) = 0$, then any $(x,s)\in P(m)$ is feasible, i.e., $(x,s)\in L^N$. In Figure~\ref{fig:example2}, we present a slice of the second-order cone at $s = 1$ for a 3D example. We have three cuts to construct $P(3)$, and we suppose $x_1,x_2\geq 0$. The red $\Delta$ corresponds to the maximal violation at $s = 1$, while $(\bar x, 1)$ is the most-violated solution from $P(m)$ with $s = 1$.

\begin{theorem}\label{thm:max_vio_reduction} 
    Suppose we add a new cut based on $y^{(m+1)}$ to $P(m)$ and get $P(m+1)$. Let $\theta_i \in [0,\pi]$ be the angle between $y^{(m+1)}$ and $y^{(i)}$, and $\bar \theta := \min_{i} \theta_i$. If $|x_j| \leq U \in \mathbb R_+$ for all $j=1,...,N-1$, and $s \leq U_r \in \mathbb R_+$, then 
    \begin{equation}\label{eq:angle_max_reduction}
       \Delta(P(m)) - \Delta(P(m+1)) = O(\bar \theta).
    \end{equation}
\end{theorem}

Theorem~\ref{thm:max_vio_reduction} quantifies the change in the proposed maximal violation of the polyhedral relaxation. In particular, if $y^{(m+1)}$ approaches an existing direction $y^{(i)}$, then the maximal reduction vanishes.

Theorems~\ref{thm:vol_reduction} and \ref{thm:max_vio_reduction} establish the same geometric phenomenon under two different measures of approximation quality. As the direction of a new cut approaches one already represented in the relaxation, both the volume and the violation bound reduction decrease at least as fast as linearly.

\subsection{Depth-Efficiency Balanced Cut Generation}

The preceding analysis identifies two different objectives in selecting a new cut. On the one hand, the cut should strongly separate the current infeasible point, favoring the conventional deepest cut. On the other hand, Theorems~\ref{thm:max_vio_reduction} and \ref{thm:vol_reduction} imply that a cut whose supporting direction is close to an existing one can have limited contribution to the polyhedral approximation. We therefore seek a supporting cut that balances separation depth against angular novelty.

We first recall the form of all nontrivial supporting hyperplanes of $L^N$.
\begin{lemma}\label{le:support_cut}
    For every supporting hyperplane that touches $L^N$ at a nonzero boundary point $(\bar y, \bar t)$, i.e., $\|\bar y\|_2 = \bar t$, it can be written in that form $\bar y^\top x \leq \|\bar y\|_2 \cdot s$.
\end{lemma}

Because positive rescaling of  $\bar y$ leaves the resulting inequality unchanged, we may normalize every supporting direction to unit norm. Henceforth, assume $\|y^{(i)}\|_2 = 1$ for all $i = 1,...,m$. To incorporate marginal efficiency, we additionally require the new direction to maintain a prescribed angular separation from all existing directions. This leads to the cut-generation problem:
\begin{subequations}\label{eq:new_cut}
    \begin{align}
        \max_y\ &y^\top \bar x - \bar s \label{eq:obj}\\
        \mbox{s.t.}\ &\|y\|_2 = 1, \label{eq:normalized}\\
        &y^\top y^{(i)} \leq \delta,\ i=1,...,m, \label{eq:cos_diff}
    \end{align}
\end{subequations}
where the objective~\eqref{eq:obj} maximizes the violation of $(\bar x, \bar s)$, \eqref{eq:normalized} normalizes $y$ to the unit sphere, and constraint~\eqref{eq:cos_diff} makes the angle between $y$ and $y^{(i)}$ is not less than $\arccos(\delta)$. If the optimal value is positive, the generated cut can separate $(\bar x, \bar s)$; otherwise, the angular requirement needs to be relaxed before a separating cut of this form can be obtained.

Although \eqref{eq:new_cut} is nonconvex because of the unit-sphere constraint, its geometry admits useful closed-form characterizations. The simplest case occurs when the conventional deepest direction already satisfies all angular constraints.

\begin{lemma}\label{le:deepest_opt}
    If $\frac{\bar{x}^\top y^{(i)}}{\|\bar x\|_2} \leq \delta$ for all $i=1,...,m$, then $\bar y := \frac{\bar{x}}{\|\bar x\|_2}$ is the optimal solution to Problem~(\ref{eq:new_cut}).
\end{lemma}

Thus, whenever the deepest direction is sufficiently separated from the existing supporting directions, the conventional deepest cut remains optimal. 

When one or more angular constraints are active, the optimal direction can still be characterized explicitly in Proposition~\ref{prop:general_opt_solution} from Appendix~\ref{app:general_proposed_cut}, but an active set for Constraint~\eqref{eq:cos_diff} of an optimal solution is required, which needs to solve a combinatorial problem at every OA iteration and would defeat the purpose of inexpensive cut generation. We therefore introduce a special choice of $\delta$ which implies an optimal solution with an analytical form.

Define $\bar \delta := \max_{i\neq j}\sqrt{\frac{{y^{(i)}}^\top y^{(j)}+ 1}{2}}$. Geometrically, if $\theta_{i,j}$ denotes the angle between $y^{(i)}$ and $y^{(j)}$, then $\bar \delta \geq \cos(\theta_{i,j}/2)$.

\begin{theorem}\label{thm:new_cut_rotating}
    Let $i^* = \arg\max_{i}\frac{\bar x^\top y^{(i)}}{\|\bar x\|_2}$, $\delta \in [\bar \delta, 1)$, and $m\geq 2$. If $\frac{\bar x^\top y^{(i^*)}}{\|\bar x\|_2} > \delta$, then the optimal solution of Problem~\eqref{eq:new_cut} is
     \begin{equation}\label{eq:rotate_cut}
         y^* = \delta y^{(i^*)} + \sqrt{1-\delta^2}\frac{(I - y^{(i^*)}{y^{(i^*)}}^\top)\bar x}{\|(I - y^{(i^*)}{y^{(i^*)}}^\top)\bar x\|_2}.
     \end{equation}
\end{theorem}

Equation~\eqref{eq:rotate_cut} has a simple geometric interpretation. Starting from the deepest direction $\frac{\bar x}{\|\bar x\|_2}$, the solution is rotated away from its closest existing supporting direction until ${y^*}^\top y^{(i^*)} = \delta$. Thus, the new cut is chosen as close as possible to the deepest direction while satisfying the prescribed angular-separation requirement.

We note that the angular-separation requirement is closely related to the common cut-management strategy of rejecting nearly parallel cuts. A detailed discussion of this connection is deferred to Appendix~\ref{app:too_parallel}. Combining Lemma~\ref{le:deepest_opt} and Theorem~\ref{thm:new_cut_rotating} yields a cut generation algorithm balancing the cut depth and the angular novelty, which is presented as Alg.~\ref{Alg:cut_generation} in Appendix~\ref{app:whole_algorithm}.

\section{Progressive-Integrality Outer Approximation}
\label{sec:progressive-oa}

Classically, in OA, the conic feasible region $\mathcal C$ is replaced by a polyhedral relaxation $P(\mathcal C) \supset \mathcal C$. Let $R(\mathcal X)$ be the continuous relaxation of $\mathcal X$, the set $\mathcal J$ includes all indices of integer variables, and $\mathcal P := P(\mathcal C)\cap R(\mathcal X)$. Therefore, in each iteration, OA solves the MILP relaxation $z_\mathcal J(\mathcal P) := \min_{x\in \mathcal P}\{c^\top x, x_i\in \mathbb Z, i\in \mathcal J\}$ and generates cuts to strengthen $\mathcal P$ \citep{castillo2016unit,dai2026solving,dai2026scheduling,liu2018global}. However, repeatedly enforcing all integrality constraints from the beginning can be unnecessarily expensive. We therefore propose a progressive-integrality OA framework that strengthens the conic approximation while gradually introducing integrality. Once the full integrality level is reached, we additionally use a fixed-integer conic subproblem to recover primal feasible solutions.

\subsection{Progressive Integrality}

For any $\mathcal I \subseteq \mathcal J$, define the partial integral relaxation
\[z_{\mathcal I}(\mathcal P) := \min_{x\in \mathcal P}\{c^\top x, x_i\in \mathbb Z, i\in \mathcal I\}.\]
The two extreme cases represent the continuous and full mixed-integer relaxations: $z_{\emptyset}(\mathcal P)$ and $z_{\mathcal J}(\mathcal P)$ separately. For the minimization problem, progressively adding integrality constraints yields nondecreasing lower bounds.

The procedure starts with $\mathcal I =\emptyset$, corresponding to an \textbf{LP stage}. A similar LP step has been considered for general convex mixed-integer nonlinear programs in \citep{kronqvist2016extended,muts2020decomposition}. Once improvement of the lower bound stalls, we progressively enlarge $\mathcal I$, corresponding to an integrality-generation stage, i.e., \textbf{IG stage}. At each iteration in this stage, we select the $p_{\mathrm{frac}}$ most fractional integer variables in the current solution, referring to $\min\{\lceil \bar x_i\rceil - \bar x_i, \bar x_i - \lfloor \bar x_i\rfloor\}$, and set them as integers. The process continues until the improvement of lower bounds stalls again or $\mathcal I = \mathcal J$, at which point the full MILP relaxation is reached. The detailed algorithm is presented as Alg.~\ref{Alg:whole_algorithm} in Appendix~\ref{app:poa}.

The progressive-integrality scheme is motivated by the different roles played by OA cuts at different integrality levels. At early stages, LP and partially integral relaxations are substantially cheaper to solve, while the globally valid cuts generated from them strengthen the continuous relaxation inherited by all subsequent MILPs. In particular, cuts from the LP stage directly tighten the root-node relaxation without being affected by competition among different integer assignments (see \textbf{replacement effect} discussed in Appendix~\ref{subsec:poa-interpretations}); thus, these cuts efficiently improve the bound at the root node in the branch-and-bound for the subsequent full mixed-integer relaxations. As integrality is progressively restored, the relaxation becomes more closely related to the full MILP, so improvements obtained at these intermediate stages are increasingly relevant to the global dual bound of OA. Also, partial mixed-integer relaxations can be interpreted as strengthening upper levels of the subsequent branch-and-bound search, which helps decrease the size of the search tree and accelerate the solution. Progressive integrality therefore aims to exploit inexpensive global strengthening early, while gradually shifting computational effort toward the integer structure as the polyhedral approximation becomes more accurate. We provide a theoretical analysis of these mechanisms, including the transfer of cut-induced bound improvements across integrality levels, the dependence of these improvements on newly generated cuts, and their implications for subsequent branch-and-bound searches, in Appendix~\ref{app:why_pioa}.

\subsection{Primal recovery by fixed-integer conic optimization}
\label{subsec:poa-recovery}

LP and IG solutions need not be feasible for the original MISOCP. Even a full-MILP solution can retain conic violations, including small violations when the cut-management procedure rejects nearly parallel candidates. We therefore retain the fixed-integer primal-recovery step. Given a full-MILP solution $\bar x\in P\cap\mathcal X$, solve
\begin{equation}
 \bar z(\bar x):=
 \min\{c^\top x:x\in\mathcal C\cap\mathcal X,
                    \ x_j=\bar x_j\ (j\in J)\}.
 \label{eq:poa-primal-recovery}
\end{equation}
All integer variables are fixed, so it is a continuous convex conic problem. If $\bar x$ makes the continuous convex program feasible, we obtain a solution $\hat x$ that is feasible for the original MISOCP; thus, $c^\top \hat x$ is a valid primal bound for OA. If the problem is infeasible, then $\bar x$ is not a feasible solution to the original MISOCP, and we continue OA without updating the primal bound.

Primal recovery is invoked when full-MILP lower-bound improvement stalls, as specified in Algorithm~\ref{Alg:whole_algorithm} of Appendix~\ref{app:whole_algorithm}. Primal recovery aims to find a feasible solution for the original MISOCP before a feasible incumbent is generated by the polyhedral relaxation, and if the gap between the primal and dual bounds is smaller than a prescribed threshold, OA can be terminated immediately.

\section{Experiments}\label{sec:experiement}


We consider two groups of instances: MISOCPs from CBLIB \citep{friberg2016cblib} and AC unit-commitment instances based on the Central Illinois 200-bus and South Carolina 500-bus systems \citep{birchfield2017}. Detailed experimental settings, parameter choices, and instance formulations are deferred to Appendix~\ref{app:setup}.

We compare: (1) OA: Alg.~\ref{Alg:oa} using the deepest cut and primal recovery, (2) OA-M: additionally uses the proposed cut-generation and management strategy, (3) POA: uses the progressive-integrality framework with the deepest cut and primal recovery, and (4) POA-M: combines progressive integrality with the proposed cut-generation strategy. For the larger instances, we also compare against the direct solution by Gurobi. All reported mean values are geometric means with a shift of one.

\subsection{Illustrative Instances}

We first use five small CBLIB instances, 50\_0\_N\_w for N $\in\{1,2,3,4,5\}$, to examine the computational mechanisms underlying the proposed methods. SCIP is used for LPs and MILPs and Mosek for SOCPs. Besides total runtime, Table~\ref{tab:small} reports the numbers of LP (\# LP), partial-MILP (\# IG), and full-MILP (\# MIP) solves. MTime and MNodes denote, respectively, the average solution time and B\&B node count of the subsequent full MILPs after the progressive stages; detailed definitions and instance-wise results are given in Appendix~\ref{app:small}.

\begin{table}[!htbp] 
    \centering 
    \small
    \setlength{\belowcaptionskip}{0pt}
    \setlength{\tabcolsep}{2.2pt}
    \begin{tabular}{l|c|ccc|cc} 
        \toprule[1pt] 
        Method &Runtime &\# LP &\# IG &\# MIP &MTime &MNodes\\ 
        \midrule
        OA &262.1 &0 &0 &6.8 &110.3 &108.3\\
        OA-M &226.7 &0 &0 &6.6 &92.1 &23.1\\
        POA &111.5 &3.8 &1 &1.7 &48.0 &4.7\\
        POA-M &64.9 &3.8 &1 &1.2 &35.7 &2.4\\
        \bottomrule[1pt] 
    \end{tabular}
    \caption{Performance and MILP-solution characteristics on illustrative instances.}
    \label{tab:small}
\end{table}

\textbf{Effect of progressive integrality.} Compared with OA, POA reduces the runtime by 57.5\% and reduces the average number of full-MILP solves from 6.8 to 1.7. The full MILPs solved after the progressive stages are also substantially easier: MTime decreases from 110.3 to 48.0, while MNodes decreases from 108.3 to 4.7. These results support the motivation in Section~\ref{sec:progressive-oa} that early LP and partial-integrality stages can strengthen the relaxation before the more expensive full-MILP phase and bring benefits to subsequent branch-and-bound search.

\textbf{Effect of cut generation and management.}  Comparing to their deepest-cut counterparts, OA-M and POA-M reduce runtime by 13.5\% and 41.8\%, respectively. They also require fewer full-MILP solves and fewer B\&B nodes. In particular, MNodes decreases from 108.3 to 23.1 for OA versus OA-M and from 4.7 to 2.4 for POA versus POA-M, providing computational evidence for the benefit of avoiding low-efficiency, nearly redundant cuts.

\subsection{Extended Instances from CBLIB}

We next consider 22 larger CBLIB instances, M\_0\_N\_w, with M $\in\{100,150,200\}$. For these instances, Gurobi is used to solve the LP and MILP subproblems and is also included as a direct MISOCP benchmark. The target optimality gap is set to $10^{-4}$. Table~\ref{tab:medium} reports the total runtime, the time required to reach optimality gaps of $1e-3$ (1e-3 Gap) and $5e-4$ (5e-4 Gap), respectively, the time when the best incumbent is obtained, and the final optimality gap. Instance-wise results are deferred to Appendix~\ref{app:medium}.

\begin{table}[!htbp] 
    \centering 
    \small
    \setlength{\belowcaptionskip}{0pt}
    \setlength{\tabcolsep}{2.2pt}
    \begin{tabular}{l|c|ccc|c} 
        \toprule[1pt] 
        Method &Runtime &1e-3 Gap &5e-4 Gap &Best &Final Gap\\ 
        \midrule
        Gurobi &237.6 &24.7 &55.0 &199.2 &9.3e-5\\
        OA &313.2 &55.4 &92.7 &265.8 &7.8e-5 \\
        OA-M &275.9 &50.5 &105.8 &231.3 &7.5e-5\\
        POA &208.6 &49.9 &51.5 &194.1 &7.4e-5 \\
        POA-M &195.9 &45.9 &48.7 &121.1 &7.3e-5 \\
        \bottomrule[1pt] 
    \end{tabular}
    \caption{Average performance on 22 extended CBLIB instances.}
    \label{tab:medium}
\end{table}

Although OA and OA-M spend longer runtimes than direct Gurobi on this set, they achieve slightly tighter final optimality gaps, indicating that their additional computational effort used to close the final gap. Introducing progressive integrality substantially improves performance, reducing the total runtime by 33.4\% and 29.0\% for POA and POA-M relative to OA and OA-M, respectively. Compared with direct Gurobi, POA and POA-M reduce runtime by 12.2\% and 17.6\%, respectively, while maintaining tighter final gaps. Furthermore, POA-M further achieves a runtime reduction of 6.1\% compared to POA and obtains its best incumbent 37.6\% earlier. These results demonstrate the effectiveness of progressive integrality on larger MISOCP instances, with the proposed cut strategy providing complementary improvements in computational efficiency.

\subsection{Large-Scale Power-System Instances}

Finally, we evaluate the methods on substantially larger, structured MISOCPs arising from AC unit commitment with SOC power-flow relaxations. We consider 24-period models based on the Central Illinois 200-bus and South Carolina 500-bus systems \citep{birchfield2017}; the formulation can be found in \cite[Problem~(1)]{dai2026scheduling}, and we set $a_{n,m} = 1$. Table~\ref{tab:large} reports total runtime to the target optimality gap of $1e-3$, time to intermediate gaps of $1e-2$ and $5e-3$, time to find the best incumbent, and the final gap.

\begin{table}[!htbp] 
    \centering 
    \small
    \setlength{\belowcaptionskip}{0pt}
    \setlength{\tabcolsep}{2pt}
    \begin{tabular}{c|l|c|ccc|c} 
        \toprule[1pt] 
        Data &Method &Runtime &1e-2 Gap &5e-3 Gap &Best &Final Gap\\ 
        \midrule
        \multirow{5}{*}{200-Bus} &Gurobi &2166 &2005 &2128 &2005 &6.8e-4\\
        &OA &1713 &1394 &1394 &1713 &9.1e-4\\
        &OA-M &1040 &776 &776 &879 &8.4e-4\\
        &POA &998 &776 &879 &998 &9.4e-4 \\
        &POA-M &655 &125 &125 &655 &9.9e-4 \\
        \midrule
        \multirow{5}{*}{500-Bus} &Gurobi &5003 &2669 &2669 &2669 &9.0e-4\\
        &OA &3634 &3106 &3106 &3634 &7.3e-4 \\
        &OA-M &2841 &2052 &2254 &2254 &8.3e-4 \\
        &POA &3060 &1436 &2145 &3060 &9.5e-4 \\
        &POA-M &1992 &1386 &1992 &1992 &9.8e-4 \\
        \bottomrule[1pt] 
    \end{tabular}
    \caption{Performance on the 200- and 500-bus AC unit-commitment instances.}
    \label{tab:large}
\end{table}

For these large-scale and structured MISOCP instances, the OA framework demonstrates a clear computational advantage over the integrated Gurobi solver. Classical OA reduces runtime by approximately 21\% and 27\% on the 200- and 500-bus instances, respectively. Consistent with the previous experiments, progressive integrality further improves OA performance, yielding additional runtime reductions of 42\% and 16\%, respectively. Notably, the proposed cut-generation strategy also provides significant improvements for both OA and POA. On the 500-bus instance, OA-M even slightly outperforms POA (which is different to the results in CPLIB), while POA-M reduces the runtime of POA by approximately 35\% on both instances. One possible explanation is that large-scale AC unit-commitment models are particularly susceptible to numerical issues, and the proposed geometry-aware cut-generation strategy helps alleviate these issues by avoiding nearly parallel or redundant cuts, thereby improving the numerical behavior of the resulting relaxations.

\section{Conclusion}

We investigated two fundamental aspects of polyhedral OA for MISOCPs: the geometric efficiency of supporting cuts and the role of integrality in constructing and solving successive relaxations. We first established theoretical bounds that quantify the marginal contribution of supporting cuts for second-order cones, showing that improvements in relaxation volume and constraint violation decrease at least linearly as a new supporting direction approaches an existing one. These geometric insights led to a closed-form cut-generation strategy that balances separation depth with directional diversity. We further developed a progressive-integrality OA framework that gradually introduces integrality constraints, allowing inexpensive LP and partially integral relaxations to strengthen the polyhedral approximation before solving full MILPs. Our theoretical analysis characterizes how cut-induced bound improvements at intermediate integrality levels transfer to the full-MILP dual bound, with stronger guarantees as integrality is progressively restored. We also provide a branch-and-bound interpretation, showing how LP and partially integral stages can strengthen the root relaxation and upper levels of the subsequent search tree, respectively.

Computational experiments on CBLIB instances and large-scale AC unit-commitment models demonstrated substantial improvements over classical OA and direct MISOCP solvers. In particular, the proposed framework reduced runtime by about 62\% and 45\% compared with classical OA on the 200- and 500-bus instances, respectively. Overall, our results demonstrate that both the geometric efficiency of cutting planes and the progressive enforcement of integrality play important roles in developing effective polyhedral OA algorithms.

\newpage

\bibliographystyle{abbrvnat}
\bibliography{references}

@ARTICLE{birchfield2017,
  author={Birchfield, Adam B. and Xu, Ti and Gegner, Kathleen M. and Shetye, Komal S. and Overbye, Thomas J.},
  journal={IEEE Transactions on Power Systems}, 
  title={Grid Structural Characteristics as Validation Criteria for Synthetic Networks}, 
  year={2017},
  volume={32},
  number={4},
  pages={3258-3265}}

@phdthesis{flores2022scheduling,
  author  = {Flores, Gonzalo E Constante},
  title   = {Scheduling of Power Units via Relaxation and Decomposition},
  school  = {The Ohio State University},
  year    = {2022}
}

@article{bezanson2017julia,
  title={Julia: A fresh approach to numerical computing},
  author={Bezanson, Jeff and Edelman, Alan and Karpinski, Stefan and Shah, Viral B},
  journal={SIAM review},
  volume={59},
  number={1},
  pages={65--98},
  year={2017},
  publisher={SIAM}
}

@article{jump,
  title={JuMP 1.0: recent improvements to a modeling language for mathematical optimization},
  author={Lubin, Miles and Dowson, Oscar and Garcia, Joaquim Dias and Huchette, Joey and Legat, Benoit and Vielma, Juan Pablo},
  journal={Mathematical Programming Computation},
  volume={15},
  pages={581--589},
  year={2023},
  publisher={Springer}
}

@misc{gurobi,
  author = {{Gurobi Optimization, LLC}},
  title = {{Gurobi Optimizer Reference Manual}},
  year = 2023,
  url = "https://www.gurobi.com"
}

@manual{mosek,
   author = "MOSEK ApS",
   title = "The MOSEK Python Fusion API manual. Version 11.0.",
   year = 2025,
   url = "https://docs.mosek.com/latest/pythonfusion/index.html"
 }

@article{constante2024security,
title = {Security-constrained unit commitment: A decomposition approach embodying Kron reduction},
journal = {European Journal of Operational Research},
volume = {319},
number = {2},
pages = {427-441},
year = {2024},
issn = {0377-2217},
author = {Gonzalo E. Constante-Flores and Antonio J. Conejo}
}

@article{bienstock2026accurate,
  title={Accurate linear cutting-plane relaxations for {ACOPF}},
  author={Bienstock, Daniel and Villagra, Matias},
  journal={Mathematical Programming Computation},
  volume={18},
  number={1},
  pages={79--133},
  year={2026},
  publisher={Springer}
}

@article{duran1986outer,
  title={An outer-approximation algorithm for a class of mixed-integer nonlinear programs},
  author={Duran, Marco A and Grossmann, Ignacio E},
  journal={Mathematical programming},
  volume={36},
  number={3},
  pages={307--339},
  year={1986},
  publisher={Springer}
}

@article{fletcher1994solving,
  title={Solving mixed integer nonlinear programs by outer approximation},
  author={Fletcher, Roger and Leyffer, Sven},
  journal={Mathematical programming},
  volume={66},
  number={1},
  pages={327--349},
  year={1994},
  publisher={Springer}
}

@article{quesada1992lp,
  title={An LP/NLP based branch and bound algorithm for convex MINLP optimization problems},
  author={Quesada, Ignacio and Grossmann, Ignacio E},
  journal={Computers \& chemical engineering},
  volume={16},
  number={10-11},
  pages={937--947},
  year={1992},
  publisher={Elsevier}
}

@article{dai2026scheduling,
  title={Scheduling Electricity Production Units to Mitigate Severe Weather Impact: An Efficient Computational Implementation},
  author={Dai, Yongzheng and Conejo, Antonio J and Qiu, Feng},
  journal = {Computers \& Operations Research},
  volume = {196},
  pages = {107668},
  year = {2026}
}

@article{dai2026solving,
  title={Solving the conic formulation of the security-constrained unit commitment problem via decomposition},
  author={Dai, Yongzheng and Conejo, Antonio J},
  journal={EURO Journal on Computational Optimization},
  pages={100136},
  year={2026},
  publisher={Elsevier}
}

@incollection{bonami2011algorithms,
  title={Algorithms and software for convex mixed integer nonlinear programs},
  author={Bonami, Pierre and Kilin{\c{c}}, Mustafa and Linderoth, Jeff},
  booktitle={Mixed integer nonlinear programming},
  pages={1--39},
  year={2011},
  publisher={Springer}
}

@article{bai2015decomposition,
  title={A decomposition method for network-constrained unit commitment with AC power flow constraints},
  author={Bai, Yang and Zhong, Haiwang and Xia, Qing and Kang, Chongqing and Xie, Le},
  journal={Energy},
  volume={88},
  pages={595--603},
  year={2015},
  publisher={Elsevier}
}

@article{liu2018global,
  title={Global solution strategies for the network-constrained unit commitment problem with {AC} transmission constraints},
  author={Liu, Jianfeng and Laird, Carl D and Scott, Joseph K and Watson, Jean-Paul and Castillo, Anya},
  journal={IEEE Transactions on Power Systems},
  volume={34},
  number={2},
  pages={1139--1150},
  year={2018},
  publisher={IEEE}
}

@ARTICLE{tuncer2023an,
  author={Tuncer, Deniz and Kocuk, Burak},
  journal={IEEE Transactions on Power Systems}, 
  title={An MISOCP-Based Decomposition Approach for the Unit Commitment Problem With AC Power Flows}, 
  year={2023},
  volume={38},
  number={4},
  pages={3388-3400}
}

@article{muckstadt1977application,
  title={An application of Lagrangian relaxation to scheduling in power-generation systems},
  author={Muckstadt, John A and Koenig, Sherri A},
  journal={Operations research},
  volume={25},
  number={3},
  pages={387--403},
  year={1977},
  publisher={INFORMS}
}

@article{zhuang2002towards,
  title={Towards a more rigorous and practical unit commitment by Lagrangian relaxation},
  author={Zhuang, Fulin and Galiana, Frank D},
  journal={IEEE Transactions on Power Systems},
  volume={3},
  number={2},
  pages={763--773},
  year={2002},
  publisher={IEEE}
}

@article{nasri2015network,
  title={Network-constrained AC unit commitment under uncertainty: A Benders’ decomposition approach},
  author={Nasri, Amin and Kazempour, S Jalal and Conejo, Antonio J and Ghandhari, Mehrdad},
  journal={IEEE transactions on power systems},
  volume={31},
  number={1},
  pages={412--422},
  year={2015},
  publisher={IEEE}
}

@article{castillo2016unit,
  title={The unit commitment problem with {AC} optimal power flow constraints},
  author={Castillo, Anya and Laird, Carl and Silva-Monroy, C{\'e}sar A and Watson, Jean-Paul and O’Neill, Richard P},
  journal={IEEE Transactions on Power Systems},
  volume={31},
  number={6},
  pages={4853--4866},
  year={2016},
  publisher={IEEE}
}

@article{kronqvist2016extended,
  title={The extended supporting hyperplane algorithm for convex mixed-integer nonlinear programming},
  author={Kronqvist, Jan and Lundell, Andreas and Westerlund, Tapio},
  journal={Journal of Global Optimization},
  volume={64},
  number={2},
  pages={249--272},
  year={2016},
  publisher={Springer}
}

@article{veinott1967supporting,
  title={The supporting hyperplane method for unimodal programming},
  author={Veinott Jr, Arthur F},
  journal={Operations Research},
  volume={15},
  number={1},
  pages={147--152},
  year={1967},
  publisher={INFORMS}
}

@incollection{land2009automatic,
  title={An automatic method for solving discrete programming problems},
  author={Land, Ailsa H and Doig, Alison G},
  booktitle={50 Years of Integer Programming 1958-2008: From the Early Years to the State-of-the-Art},
  pages={105--132},
  year={2009},
  publisher={Springer}
}

@article{achterberg2005branching,
  title={Branching rules revisited},
  author={Achterberg, Tobias and Koch, Thorsten and Martin, Alexander},
  journal={Operations Research Letters},
  volume={33},
  number={1},
  pages={42--54},
  year={2005},
  publisher={Elsevier}
}

@inproceedings{achterberg2009hybrid,
  title={Hybrid branching},
  author={Achterberg, Tobias and Berthold, Timo},
  booktitle={International Conference on Integration of Constraint Programming, Artificial Intelligence, and Operations Research},
  pages={309--311},
  year={2009},
  organization={Springer}
}

@InProceedings{scip1,
 Title = {Constraint Integer Programming: A New Approach to Integrate CP and MIP},
 Author = {Achterberg, Tobias and Berthold, Timo and Koch, Thorsten and Wolter, Kati},
 Booktitle = {Integration of AI and OR Techniques in Constraint Programming for Combinatorial Optimization Problems},
 Year = {2008},
 Address = {Berlin, Heidelberg},
 Editor = {Perron, Laurent and Trick, Michael A.},
 Pages = {6--20},
 Publisher = {Springer Berlin Heidelberg}
}

@article{scip2,
 Title = {SCIP: solving constraint integer programs},
 Author = {Achterberg, Tobias},
 Journal = {Mathematical Programming Computation},
 Year = {2009},
 Month = {Jul},
 Number = {1},
 Pages = {1--41},
 Volume = {1},
 Day = {01}
}

@article{friberg2016cblib,
  title={CBLIB 2014: a benchmark library for conic mixed-integer and continuous optimization},
  author={Friberg, Henrik A},
  journal={Mathematical Programming Computation},
  volume={8},
  number={2},
  pages={191--214},
  year={2016},
  publisher={Springer}
}

@article{sahinidis1996baron,
  title={BARON: A general purpose global optimization software package},
  author={Sahinidis, Nikolaos V},
  journal={Journal of global optimization},
  volume={8},
  number={2},
  pages={201--205},
  year={1996},
  publisher={Springer}
}

@article{muts2020decomposition,
  title={The decomposition-based outer approximation algorithm for convex mixed-integer nonlinear programming},
  author={Muts, Pavlo and Nowak, Ivo and Hendrix, Eligius MT},
  journal={Journal of Global Optimization},
  volume={77},
  number={1},
  pages={75--96},
  year={2020},
  publisher={Springer}
}

@incollection{benson2013mixed,
  title={Mixed-integer second-order cone programming: A survey},
  author={Benson, Hande Y and Sa{\u{g}}lam, {\"U}mit},
  booktitle={Theory driven by influential applications},
  pages={13--36},
  year={2014},
  publisher={Informs}
}

@TECHREPORT{SAK:assortment,
  AUTHOR       = {A. \c{S}en and A. Atamt{\"u}rk and P. Kaminsky},
  TITLE        = {A Conic Integer Programming Approach to Constrained 
                  Assortment Optimization under the Mixed Multinomial 
                  Logit Model},
  YEAR         = {2015},
  INSTITUTION  = {IEOR, University of California--Berkeley},
  TYPE         = {Research Report},
  NUMBER = {BCOL.15.06},
  MONTH = {October},
}

@article{ghaoui2003worst,
  title={Worst-case value-at-risk and robust portfolio optimization: A conic programming approach},
  author={Ghaoui, Laurent El and Oks, Maksim and Oustry, Francois},
  journal={Operations research},
  volume={51},
  number={4},
  pages={543--556},
  year={2003},
  publisher={INFORMS}
}

@article{byeon2019unit,
  title={Unit commitment with gas network awareness},
  author={Byeon, Geunyeong and Van Hentenryck, Pascal},
  journal={IEEE Transactions on Power Systems},
  volume={35},
  number={2},
  pages={1327--1339},
  year={2019},
  publisher={IEEE}
}

@article{kucukyavuz2023consistent,
  title={Consistent second-order conic integer programming for learning Bayesian networks},
  author={Kucukyavuz, Simge and Shojaie, Ali and Manzour, Hasan and Wei, Linchuan and Wu, Hao-Hsiang},
  journal={Journal of Machine Learning Research},
  volume={24},
  number={322},
  pages={1--38},
  year={2023}
}

@article{ben2001polyhedral,
  title={On polyhedral approximations of the second-order cone},
  author={Ben-Tal, Aharon and Nemirovski, Arkadi},
  journal={Mathematics of Operations Research},
  volume={26},
  number={2},
  pages={193--205},
  year={2001},
  publisher={INFORMS}
}

@article{vielma2008lifted,
  title={A lifted linear programming branch-and-bound algorithm for mixed-integer conic quadratic programs},
  author={Vielma, Juan Pablo and Ahmed, Shabbir and Nemhauser, George L},
  journal={INFORMS Journal on Computing},
  volume={20},
  number={3},
  pages={438--450},
  year={2008},
  publisher={INFORMS}
}

@article{atamturk2010conic,
  title={Conic mixed-integer rounding cuts},
  author={Atamt{\"u}rk, Alper and Narayanan, Vishnu},
  journal={Mathematical programming},
  volume={122},
  number={1},
  pages={1--20},
  year={2010},
  publisher={Springer}
}

@article{lubin2018polyhedral,
  title={Polyhedral approximation in mixed-integer convex optimization},
  author={Lubin, Miles and Yamangil, Emre and Bent, Russell and Vielma, Juan Pablo},
  journal={Mathematical Programming},
  volume={172},
  number={1},
  pages={139--168},
  year={2018},
  publisher={Springer}
}

@article{coey2020outer,
  title={Outer approximation with conic certificates for mixed-integer convex problems},
  author={Coey, Chris and Lubin, Miles and Vielma, Juan Pablo},
  journal={Mathematical Programming Computation},
  volume={12},
  number={2},
  pages={249--293},
  year={2020},
  publisher={Springer}
}

@article{kelley1960cutting,
  title={The cutting-plane method for solving convex programs},
  author={Kelley, Jr, James E},
  journal={Journal of the society for Industrial and Applied Mathematics},
  volume={8},
  number={4},
  pages={703--712},
  year={1960},
  publisher={SIAM}
}

@article{serrano2020relation,
  title={On the relation between the extended supporting hyperplane algorithm and Kelley’s cutting plane algorithm: F. Serrano et al.},
  author={Serrano, Felipe and Schwarz, Robert and Gleixner, Ambros},
  journal={Journal of Global Optimization},
  volume={78},
  number={1},
  pages={161--179},
  year={2020},
  publisher={Springer}
}

@article{abhishek2010filmint,
  title={FilMINT: An outer approximation-based solver for convex mixed-integer nonlinear programs},
  author={Abhishek, Kumar and Leyffer, Sven and Linderoth, Jeff},
  journal={INFORMS Journal on computing},
  volume={22},
  number={4},
  pages={555--567},
  year={2010},
  publisher={INFORMS}
}

@article{turner2023adaptive,
  title={Adaptive cut selection in mixed-integer linear programming},
  author={Turner, Mark and Koch, Thorsten and Serrano, Felipe and Winkler, Michael},
  journal={Open Journal of Mathematical Optimization},
  volume={4},
  pages={1--28},
  year={2023}
}

@article{warme2026quantitativeII,
  title={Quantitative indicators for strength of inequalities with respect to a polyhedron, Part II: Applications and computational evidence},
  author={Warme, David M},
  journal={Mathematical Programming},
  pages={1--35},
  year={2026},
  publisher={Springer}
}

@article{warme2026quantitativeI,
  title={Quantitative Indicators for Strength of Inequalities with Respect to a Polyhedron},
  author={Warme, David M},
  journal={Mathematical Programming},
  pages={1--39},
  year={2026},
  publisher={Springer}
}

@article{burke1993weak,
  title={Weak sharp minima in mathematical programming},
  author={Burke, James V and Ferris, Michael C},
  journal={SIAM Journal on Control and Optimization},
  volume={31},
  number={5},
  pages={1340--1359},
  year={1993},
  publisher={SIAM}
}

@article{abbasi2021strongly,
  title={Strongly regular points of mappings},
  author={Abbasi, Malek and Th{\'e}ra, Michel},
  journal={Fixed Point Theory and Algorithms for Sciences and Engineering},
  volume={2021},
  number={1},
  pages={14},
  year={2021},
  publisher={Springer}
}

\newpage
\appendix

\section{Proofs}
\subsection{Proofs of Proposition~\ref{prop:deepest_cut}}\label{app:prop:deepest_cut}
\begin{proof}
    From \cite[Proposition 5]{bienstock2026accurate}, the closed form of the deepest cut is $(x^{(1)})^\top x\leq \|x^{(1)}\|s$ after scaling by $\frac{1}{\sqrt{2}}$, and the projection from $(x^{(1)}, s^{(1)})$ onto $L^N$ is
    \begin{equation}\label{eq:projection_point}
        (y^{(1)}, t^{(1)}) := (\frac{t^{(1)}x^{(1)}}{\|x^{(1)}\|_2}, \frac{\|x^{(1)}\|_2 + s^{(1)}}{2}).
    \end{equation}
    From $(y^{(1)}, t^{(1)}) \in \partial(L^N)$, there is $\|y^{(1)}\|_2 = t^{(1)}$. Plugging it into $y^{(1)} = \frac{t^{(1)}x^{(1)}}{\|x^{(1)}\|_2}$, we have \[\frac{y^{(1)}}{\|y^{(1)}\|_2} = \frac{x^{(1)}}{\|x^{(1)}\|_2};\]
    thus the deepest cut is equivalent to $(y^{(1)})^\top x\leq \|y^{(1)}\|s$. 
    
    Plugging \eqref{eq:projection_point} into the cut, we find that
    \[(y^{(1)})^\top y^{(1)} = \|y^{(1)}\|t^{(1)};\]
    thus the deepest cut is a supporting cut for $L^N$.

    Additionally, let $g(x, s) := \|x\|_2 - s$. Then
    \[\begin{aligned}
        &g(y^{(1)}, t^{(1)}) + \nabla_x g(y^{(1)}, t^{(1)})^\top (x - y^{(1)}) + \nabla_s g(y^{(1)}, t^{(1)})^\top (s - t^{(1)}) \leq 0\\
        \iff &\|y^{(1)}\|_2 - t^{(1)} + (\frac{y^{(1)}}{\|y^{(1)}\|_2})^\top(x-y^{(1)}) -(s - t^{(1)}) \leq 0\\
        \iff &(y^{(1)})^\top x - t^{(1)}s \leq 0.
    \end{aligned}\]
    Combining $y^{(1)} = \frac{t^{(1)}x^{(1)}}{\|x^{(1)}\|_2}$ and $\frac{y^{(1)}}{\|y^{(1)}\|_2} = \frac{x^{(1)}}{\|x^{(1)}\|_2}$, we have $t^{(1)} = \|y^{(1)}\|_2$. Therefore, the deepest cut is a gradient cut based on $(y^{(1)}, t^{(1)})$. \qed
\end{proof}

\subsection{Proofs of Theorem~\ref{thm:dist_bound}}\label{app:thm:dist_bound}
\begin{proof}
    For $(x^{(2)}, s^{(2)}) \in L^N$, Equations~\eqref{eq:euclid_bound} and \eqref{eq:angle_bound} hold since the RHS is nonnegative. We discuss $(x^{(2)}, s^{(2)}) \notin L^N$ in the proof.
    
    From $(x^{(2)}, s^{(2)}) \in P$, there is $(y^{(1)}/\|y^{(1)}\|_2)^\top x^{(2)} \leq s^{(2)}$. Let $e = y^{(1)}/\|y^{(1)}\|_2$, in which $e^\top e = 1$. We do an orthogonal decomposition: 
    \[x^{(2)}:= (e^\top x^{(2)})e + (x^{(2)}-(e^\top x^{(2)})e),\]
    where $e^\top (x^{(2)}-(e^\top x^{(2)})e) = 0$. We split the proof into 2 cases.

    \textbf{First Case: $e^\top x^{(2)} \geq 0$.} Then we have $|e^\top x^{(2)}| = e^\top x^{(2)} \leq s^{(2)}$, and it follows
    \begin{equation}\label{eq:x_2_norm}
        \|x^{(2)}\|_2^2 = (e^\top x^{(2)})^2 + \|x^{(2)}-(e^\top x^{(2)})e\|_2^2 \leq (s^{(2)})^2 + \|x^{(2)}-(e^\top x^{(2)})e\|_2^2.
    \end{equation}
    From \cite[Proposition 5]{bienstock2026accurate}, the Euclidean distance between $(x^{(2)}, s^{(2)})\not\in L^N$ and $L^N$ satisfies
    \[\mathrm{dist}((x^{(2)}, s^{(2)}), L^N)^2 = \frac{(\|x^{(2)}\|_2 - s^{(2)})^2}{2}.\]
    Plugging Equation~\eqref{eq:x_2_norm} into the right-hand-side,
    \begin{equation}\label{eq:dist_bound}
        \begin{aligned}
            \mathrm{dist}((x^{(2)}, s^{(2)}), L^N)^2 
            \leq &\frac{(\sqrt{(s^{(2)})^2 + \|x^{(2)}-(e^\top x^{(2)})e\|_2} - s^{(2)})^2}{2}\\
            \leq &\frac{\|x^{(2)}-(e^\top x^{(2)})e\|_2^2}{2}
        \end{aligned}
    \end{equation}
    On the other hand, $y^{(1)} = \|y^{(1)}\|_2e$. Therefore,
    \[x^{(2)} - y^{(1)} = (e^\top x^{(2)} - \|y^{(1)}\|_2)e + (x^{(2)}-(e^\top x^{(2)})e);\]
    thus,
    \[\|x^{(2)} - y^{(1)}\|_2^2 \geq \|x^{(2)}-(e^\top x^{(2)})e\|_2^2.\]
    Plugging Equation~\eqref{eq:dist_bound} into the inequality, we have
    \[\mathrm{dist}((x^{(2)}, s^{(2)}), L^N)^2 \leq \frac{\|x^{(2)}-(e^\top x^{(2)})e\|_2^2}{2} \leq \frac{\|x^{(2)} - y^{(1)}\|_2^2}{2}.\]

    \textbf{Second Case: $e^\top x^{(2)} < 0$.} We have
    \begin{equation*}
        \begin{aligned}
            \|x^{(2)} - y^{(1)}\|_2^2 &= \|(e^\top x^{(2)} - \|y^{(1)}\|_2)e + (x^{(2)}-(e^\top x^{(2)})e\|_2\\
            &= (e^\top x^{(2)} - \|y^{(1)}\|_2)^2 + \|x^{(2)}-(e^\top x^{(2)})e\|_2^2.
        \end{aligned}
    \end{equation*}
    Therefore,
    \begin{equation*}
        \begin{aligned}
            \|x^{(2)} - y^{(1)}\|_2^2 - \|x^{(2)}\|_2^2 &= (e^\top x^{(2)} - \|y^{(1)}\|_2)^2 - (e^\top x^{(2)})^2\\
            &=\|y^{(1)}\|_2^2 - 2(e^\top x^{(2)})\|y^{(1)}\|_2.
        \end{aligned}
    \end{equation*}
    Since $e^\top x^{(2)} < 0$, it follows $\|x^{(2)} - y^{(1)}\|_2^2 - \|x^{(2)}\|_2^2 > 0$. Consider
    \[\mathrm{dist}((x^{(2)}, s^{(2)}), L^N)^2 = \frac{(\|x^{(2)}\|_2 - s^{(2)})^2}{2} \leq \frac{\|x^{(2)}\|_2^2}{2},\]
    we have
    \[\mathrm{dist}((x^{(2)}, s^{(2)}), L^N)^2 \leq \frac{\|x^{(2)}\|_2^2}{2} \leq \frac{\|x^{(2)} - y^{(1)}\|_2^2}{2}.\]

    Therefore, we prove Equation~\eqref{eq:euclid_bound} for the two cases.
    
    Furthermore, the angle $\theta$ between $y^{(1)}$ and $x^{(2)}$ such that
    \[\cos(\theta) = \frac{(y^{(1)})^\top x^{(2)}}{\|y^{(1)}\|_2\|x^{(2)}\|_2}.\]
    Since $(x^{(2)}, s^{(2)}) \in P$, $(y^{(1)}/\|y^{(1)}\|_2)^\top x^{(2)} \leq s^{(2)}$. Therefore, $\cos(\theta) \leq \frac{s^{(2)}}{\|x^{(2)}\|_2}$.
    Then the violation
    \begin{equation}\label{eq:thm1_angle_proof}
        \|x^{(2)}\|_2 - s^{(2)} \leq \|x^{(2)}\|_2 - \cos(\theta)\|x^{(2)}\|_2 =\left(1 - \cos(\theta)\right)\|x^{(2)}\|_2.
    \end{equation}
    Plugging $\mathrm{dist}((x^{(2)}, s^{(2)}), L^N) = \frac{\sqrt{2}(\|x^{(2)}\|_2 - s^{(2)})}{2}$ into the inequality, we have
    \[\mathrm{dist}((x^{(2)}, s^{(2)}), L^N) \leq \frac{\sqrt{2}(1-\cos(\theta))\|x^{(2)}\|_2}{2}.\]
    \qed
\end{proof}

\subsection{Proofs of Theorem~\ref{thm:vol_reduction}}\label{app:thm:vol_reduction}

\begin{proof}
    We begin with a slice of the space where $s = r \leq U_r$. Then the slice of the original relaxation is
    \[P_m(r) := \{x\in \mathbb R^{N-1}: |x_i| \leq U, u_i^\top x \leq r, \forall i=1,...,m\},\]
    where $u_i = \frac{y^{(i)}}{\|y^{(i)}\|} \in \mathbb S^{N-2}$, and $\mathbb S^{N-2}$ is the unit sphere in $\mathbb R^{N-1}$. And
    \[P_{m+1}(r) := P_m(r) \cap \{x\in \mathbb R^{N-1}: u_{m+1}^\top x \leq r\},\]
    i.e., the slice of the strengthened relaxation. Let $D(r) := P_{m}(r)\backslash P_{m+1}(r)$ be the region removed by the new cut. We aim to derive the upper bound for $\mathrm{Vol}(D(r))$.
    
    Let $\bar u\in\{u_1, ..., u_m\}$ be the unit vector whose angle to $u_{m+1}$ is $\bar \theta$. 

    For every $x\in D(r)$, there is
    \[\bar u^\top x \leq r\ \land\ u_{m+1}^\top x > r.\]
    \[\implies r < u_{m+1}^\top x = \bar u^\top x + (u_{m+1} - \bar u)^\top x.\]
    Rearranging gives
    \[0\leq r - \bar u^\top x < (u_{m+1} - \bar u)^\top x.\]
    By the Cauchy–Schwarz inequality,
    \[r - \bar u^\top x < \|u_{m+1} - \bar u\|_2 \|x\|_2.\]
    Let 
    \[R(r) := \max\{\|x\|_2: |x_i| \leq U, u_i^\top x \leq r, \forall i=1,...,m\}\]
    be the upper bound of $\|x\|_2$. Since $x_i$ are bounded, $R(r)<+\infty$. Since $\bar u$ and $u_{m+1}$ are unit vectors separated by angle $\bar\theta$,
    \[\|u_{m+1} - \bar u\|_2 = 2\sin(\frac{\bar\theta}{2}).\]
    Therefore,
    \[r - \bar u^\top x < \|u_{m+1} - \bar u\|_2 \|x\|_2 \leq 2R(r)\sin(\frac{\bar\theta}{2}).\]
    It follows that every $x\in D(r)$ satisfies
    \[r - 2R(r)\sin(\frac{\bar\theta}{2})< \bar u^\top x \leq r.\]
    Thus,
    \[D(r)\subseteq P_m(r) \cap \{x: r - 2R(r)\sin(\frac{\bar\theta}{2})< \bar u^\top x \leq r\}.\]
    The removed region is therefore contained in a slab adjacent to the existing hyperplane $\bar u^\top x = r$, whose thickness is $2R(r)\sin(\frac{\bar\theta}{2})$. Let
    \[\bar A(r) := \sup_{t\in\mathbb R}\mathrm{Vol}(P_m(r)\cap\{x:\bar u^\top x = t\}),\]
    the upper bound of the $(N-2)$-dimensional volume of every cross-section of $P_m(r)$ perpendicular to $\bar u$. Since $x$ is bounded, $\bar A(r) <\infty$. As a result,
    \[\mathrm{Vol}(D(r)) \leq 2R(r)\bar A(r)\sin\frac{\bar\theta}{2}.\]
    
    Finally,
    \begin{equation*}
        \begin{aligned}
            \mathrm{Vol}(P(m)) - \mathrm{Vol}(P(m+1)) &= \int_0^{U_r} \mathrm{Vol}(D(r))\, dr\\ 
            &\leq \int_0^{U_r} 2R(r)\bar A(r)\sin\frac{\bar\theta}{2}\, dr\\
            &= \left(\int_0^{U_r} R(r)\bar A(r)\, dr\right)\cdot 2\sin\frac{\bar\theta}{2}\\
    (2\sin\frac{\bar\theta}{2}\leq \bar\theta)\quad        &\leq \left(\int_0^{U_r} 2R(r)\bar A(r)\, dr\right)\bar\theta.
        \end{aligned}
    \end{equation*}
    Since both $R(r)$ and $\bar{A}(r)$ can be uniformly bounded over $r\in [0, U_r]$, the first term 
    \[\int_0^{U_r} 2R(r)\bar A(r)\, dr <\infty;\] 
    therefore, $\mathrm{Vol}(P(m)) - \mathrm{Vol}(P(m+1)) = O(\bar\theta)$.
    \qed
\end{proof}

\subsection{Proofs of Theorem~\ref{thm:max_vio_reduction}}\label{app:thm:max_vio_reduction}
\begin{proof}
    For any fixed $x$, define the smallest $s$ allowed by the added $m$ cuts:
    \[h_m(x) = \max\{0,\max_{i=1,...,m}\{(y^{(i)}/\|y^{(i)}\|_2)^\top x\}\}.\]
    Then 
    $$\Delta(P(m)) = \max_{(x,s)\in P(m)}\{[\|x\|_2 - h_m(x)]_+\}.$$
    After adding the cut based on $y^{(m+1)}$,
    \[h_{m+1}(x) =\max\{h_{m}(x), (y^{(m+1)}/\|y^{(m+1)}\|_2)^\top x\}.\]
    Let $j := \arg\min_{i=1,...,m}\theta_i$, i.e., $\theta_j = \bar\theta$. Since $h_{m}(x) \geq (y^{(j)}/\|y^{(j)}\|_2)^\top x$,
    \[\begin{aligned}
        0 &\leq h_{m+1}(x) - h_m(x)\\
        &= [(y^{(m+1)}/\|y^{(m+1)}\|_2)^\top x - h_m(x)]_+\\
        &\leq [(y^{(m+1)}/\|y^{(m+1)}\|_2- y^{(j)}/\|y^{(j)}\|_2)^\top x]_+\\
        &\leq \|(y^{(m+1)}/\|y^{(m+1)}\|_2- y^{(j)}/\|y^{(j)}\|_2)\|_2 \|x\|_2.
    \end{aligned}\]
    Because both directions are unit vectors,
    \[\|(y^{(m+1)}/\|y^{(m+1)}\|_2- y^{(j)}/\|y^{(j)}\|_2)\|_2 = \sqrt{2-2\cos(\bar\theta)} = 2\sin(\bar\theta/2).\]
    Hence,
    \[h_{m+1}(x) - h_m(x) \leq U^{N-1}\cdot 2\sin(\bar\theta/2).\]
    Since $t \to [t]_+$ is 1-Lipschitz,
    \[[\|x\|_2 - h_m(x)]_+ - [\|x\|_2 - h_{m+1}(x)]_+\leq  h_{m+1}(x) - h_m(x)\leq U^{N-1}\cdot 2\sin(\bar\theta/2).\]
    Finally,
    \[\begin{aligned}
        \Delta(P(m)) - \Delta(P(m+1)) &= \max_x[\|x\|_2 - h_m(x)]_+ - \max_x[\|x\|_2 - h_{m+1}(x)]_+\\
        &\leq \max_x\{[\|x\|_2 - h_m(x)]_+ - [\|x\|_2 - h_{m+1}(x)]_+\}\\
        &\leq U^{N-1}\cdot 2\sin(\bar\theta/2)\\
        &= O(\bar \theta).
    \end{aligned}\]
    \qed
\end{proof}

\subsection{Proofs of Lemma~\ref{le:support_cut}}\label{app:le:support_cut}

\begin{proof}
    Let $g(x, s) = \|x\|_2 - s$. If $x\neq 0$, then
    \[\nabla g(x, s) = (\frac{x}{\|x\|_2}, -1).\]
    Then the supporting inequality at $(\bar y, \bar t)$ is
    \[g(\bar y, \bar t) + \nabla_x g(\bar y, \bar t)(x - \bar y) +\nabla_s g(\bar y, \bar t) (s - \bar t)\leq 0\]
    \[\implies g(\bar y, \bar t) + \frac{\bar y}{\|\bar y\|_2}(x - \bar y) - (s - \bar t)\leq 0.\]
    Since $g(\bar y, \bar t) = \|\bar y\|_2 - \bar t = 0$, this becomes
    \[\frac{\bar y^\top x}{\|\bar y\|_2} - s\leq 0 \iff \bar y^\top x \leq \|\bar y\|_2 \cdot s.\]
    \qed
\end{proof}

\subsection{Proofs of Lemma~\ref{le:deepest_opt}}\label{app:le:deepest_opt}
\begin{proof}
    If $\frac{\bar{x}^\top y^{(i)}}{\|\bar x\|_2} \leq \delta$ for all $i=1,...,m$, then $\bar y = \frac{\bar{x}}{\|\bar x\|_2}$ is feasible. Since $\frac{\bar{x}}{\|\bar x\|_2}$ defines the deepest cut, it is the optimal solution. \qed
\end{proof}

\subsection{Proof of Theorem~\ref{thm:new_cut_rotating}}\label{app:thm:new_cut_rotating}

\begin{proof}
    We consider a restricted problem to \eqref{eq:new_cut}:
    \begin{equation*}
        \begin{aligned}
            \max_y\ &\bar{x}^\top y - \bar s \\
            \mbox{s.t.}\ &\|y\|_2 = 1,\\
            &y^\top y^{(i^*)} \leq \delta.
        \end{aligned}
    \end{equation*}
    Since $\frac{(\bar x)^\top y^{(i^*)}}{\|\bar x\|_2} > \delta$, the unconstrained optimizer $\frac{\bar x}{\|\bar x\|_2}$ violates the only angular constraint. Hence, this constraint is active at an optimal solution of the one-constraint relaxation. Therefore, from Proposition~\ref{prop:general_opt_solution}, we have a candidate of the optimal solution
    \[y^* = \delta y^{(i^*)} + \sqrt{1-\delta^2}\frac{(I - y^{(i^*)}{y^{(i^*)}}^\top) x}{\|(I - y^{(i^*)}{y^{(i^*)}}^\top)\bar x\|_2},\]
    and $\angle (y^*, y^{(i^*)}) = \arccos(\delta)$.

    Then we verify the feasibility of $y^*$ for other constraints in Problem~\eqref{eq:new_cut}. From the definition of $\bar \delta$, for any $i\neq i^*$, we have
    \[{y^{(i)}}^\top y^{(i^*)} \leq 2\delta^2 - 1 = \cos(2\arccos(\delta)),\]
    \[\implies \angle(y^{(i)}, y^{(i^*)}) \geq 2\arccos(\delta).\]
    By the reverse triangle inequality for angles on the unit sphere,
    \[\angle (y^*, y^{(i)}) \geq \angle (y^{(i^*)} , y^{(i)}) - \angle (y^*, y^{(i^*)}) \geq \arccos(\delta);\]
    thus, ${y^*}^\top y^{(j)} = \cos(\angle (y^*, y^{(j)})) \leq \delta$. Therefore, $y^*$ is feasible in Problem~(\ref{eq:new_cut}), and hence it is optimal.
\end{proof}

\section{Tightness of Theorem~\ref{thm:vol_reduction}}\label{app:extend_vol_reduction}

Theorem~\ref{thm:vol_reduction} shows that the volume removed by a new supporting
cut is $O(\bar\theta)$, where $\bar\theta$ is the minimum angular distance
between the new supporting direction and the existing supporting directions.
We next show that this bound is tight under a local nondegeneracy condition.

Consider a family of new unit supporting directions $u_{m+1}$
approaching an existing unit direction $\bar u\in\{u_1,\ldots,u_m\}$.
For sufficiently small $\bar\theta>0$, write
\[u_{m+1}= \cos(\bar\theta)\bar u + \sin(\bar\theta)w, \]
where $\|w\|_2=1$ and $w^\top \bar u=0$.

For a fixed slice $s=r$, define
\[P_m(r):=\left\{x\in\mathbb R^{N-1}:|x_j|\le U,\ u_i^\top x\le r,\quad i=1,\ldots,m\right\},\]
and
\[F(r):=P_m(r)\cap\{x:\bar u^\top x=r\}.\]
We impose the following local nondegeneracy condition. Suppose there exists an
interval $I\subset(0,U_r)$ with positive length and constants
$a,\eta,\rho>0$ such that, for every $r\in I$, there exists a measurable set
\[K(r)\subseteq \operatorname{ri}(F(r))\]
satisfying
\[\mathcal H^{N-2}(K(r))\ge a,\]
\[w^\top y\ge \eta,\quad\forall y\in K(r),\]
and
\[y-p\bar u\in P_m(r),\quad \forall y\in K(r),\  0\le p\le \rho.\]
Here $\mathcal H^{N-2}$ denotes the $(N-2)$-dimensional Hausdorff measure.
Geometrically, this condition requires that a uniformly positive portion of
the facet $F(r)$ is exposed in the direction $w$, and that the relaxation has
a uniformly positive inward neighborhood behind that portion of the facet.

Under this condition, we have
\begin{equation}\label{eq:vol_reduction_linear}
    \operatorname{Vol}(P(m))-\operatorname{Vol}(P(m+1))=\Theta(\bar\theta), \quad \bar\theta\to 0.
\end{equation}

\begin{proof}
Theorem~\ref{thm:vol_reduction} already gives
\[\operatorname{Vol}(P(m))-\operatorname{Vol}(P(m+1))=O(\bar\theta).\]
It remains to prove the matching lower bound.

Fix $r\in I$ and $y\in K(r)$. Consider a point immediately inside the existing
facet,
\[x=y-p\bar u,\quad p\ge0.\]
Since $y\in F(r)$, we have
\[\bar u^\top y=r.\]
The violation of the new cut at $x$ is
\[u_{m+1}^\top x-r.\]
Using
\[u_{m+1}=\cos(\bar\theta)\bar u+\sin(\bar\theta)w\]
and $w^\top \bar u=0$, we obtain
\[\begin{aligned}
u_{m+1}^\top x-r
&=\left(\cos(\bar\theta)\bar u+\sin(\bar\theta)w\right)^\top (y-p\bar u)-r  \\
&= r(\cos(\bar\theta)-1) + \sin(\bar\theta)w^\top y - p\cos(\bar\theta).
\end{aligned}\]
Since $w^\top y\ge \eta$, and since $r\le U_r$, for sufficiently small
$\bar\theta>0$ we have
\[\sin(\bar\theta)\ge \frac{\bar\theta}{2},\quad 1-\cos(\bar\theta)\le \frac{\bar\theta^2}{2},\quad \cos(\bar\theta)\le 1.\]
Therefore,
\[u_{m+1}^\top x-r\ge -\frac{U_r}{2}\bar\theta^2+ \frac{\eta}{2}\bar\theta-p.\]
Now choose
\[0\le p\le \frac{\eta}{4}\bar\theta.\]
Then
\[u_{m+1}^\top x-r\ge\frac{\eta}{4}\bar\theta-\frac{U_r}{2}\bar\theta^2.\]
Thus, for all sufficiently small $\bar\theta>0$, $u_{m+1}^\top x>r$. Hence $x$ violates the new supporting cut.

At the same time, by the local nondegeneracy condition, if $\bar\theta$ is
sufficiently small so that $\frac{\eta}{4}\bar\theta\le \rho$,
then $y-p\bar u\in P_m(r)$ for every $y\in K(r)$ and every $0\le p\le \frac{\eta}{4}\bar\theta$. Consequently, the set
\[E_{\bar\theta}(r):=\left\{y-p\bar u:y\in K(r),\;0\le p\le \frac{\eta}{4}\bar\theta\right\}\]
is contained in the part of $P_m(r)$ removed by the new cut.

Because $\bar u$ is normal to the facet $F(r)$, the thickness of the set $E_{\bar\theta}(r)$ is $\eta\bar\theta/4$ over $K(r)$. Therefore,
\[\begin{aligned}
\operatorname{Vol}\left(P_m(r)\setminus P_{m+1}(r)\right) &\ge \operatorname{Vol}\left(E_{\bar\theta}(r)\right)  \\
&=\frac{\eta}{4}\bar\theta\cdot \mathcal H^{N-2}(K(r))  \\
&\ge \frac{a\eta}{4}\bar\theta.
\end{aligned}\]
Integrating over $r\in I$ gives
\[\begin{aligned}
\operatorname{Vol}(P(m))-\operatorname{Vol}(P(m+1)) &= \int_0^{U_r}\operatorname{Vol}_{N-1}\left(P_m(r)\setminus P_{m+1}(r)\right)\,dr  \\
&\ge\int_I\frac{a\eta}{4}\bar\theta\,dr  \\
&=\frac{a\eta |I|}{4}\bar\theta.
\end{aligned}\]
Thus,
\[\operatorname{Vol}(P(m))-\operatorname{Vol}(P(m+1))=\Omega(\bar\theta).\]
Combining this lower bound with the upper bound
\[\operatorname{Vol}(P(m))-\operatorname{Vol}(P(m+1))=O(\bar\theta)\]
from Theorem~\ref{thm:vol_reduction}, we obtain
\[\operatorname{Vol}(P(m))-\operatorname{Vol}(P(m+1))=\Theta(\bar\theta).\]
\end{proof}

\section{Explicit Description for the Proposed Cut in the General Situation}\label{app:general_proposed_cut}

\begin{proposition}\label{prop:general_opt_solution}
    Suppose $S\subseteq \{1,...,m\}$ is the active set for Constraint~\eqref{eq:cos_diff} of an optimal solution to Problem~(\ref{eq:new_cut}), and vectors from $\{y^{(i)}: i\in S\}$ are linearly independent. Let $A_S = [y^{(i)}]_{i\in S}$, $\delta_S = [\delta, ..., \delta]\in\mathbb R^{|S|}$, $p_S = A_S(A_S^\top A_s)^{-1}\delta_S$, and $P_S := I - A_S (A_S^\top A_S)^{-1} A_S^\top$. If $P_S \bar x \neq 0$, then the solution
    \begin{equation}\label{eq:general_S_cut}
        y^* = p_S + \sqrt{1-\|p_S\|_2^2}\frac{P_S\bar x}{\|P_S\bar x\|_2},
    \end{equation}
    is optimal if it satisfies \eqref{eq:cos_diff}; if $P_S \bar x = 0$, every feasible $y = p_S + z$ with $z\in\mathrm{Null}(A_S^\top)$ and $\|y\|_2 = 1$ is optimal.
\end{proposition}

\begin{proof}
    As $S$ is the active set to an optimal solution, for any candidate $\bar y$ of the optimal solution, we have
    \[{y^{(i)}}^\top \bar y = \delta,\ \forall i\in S.\]
    Let $A_S = [y^{(i)}]_{i\in S}$ and $\delta_S = [\delta, ..., \delta]\in\mathbb R^{|S|}$. Since $y^{(i)}$ is linearly independent of each other, the matrix $A_s$ has full column rank. We have
    \[A_s^\top \bar y = \delta_S.\]
    Let the vector $p_S = A_S(A_S^\top A_s)^{-1}\delta_S$, then
    \[\bar y = p_S + z,\]
    where $z \in \mbox{Null}(A_S^\top)$. From $\|\bar y\|_2 = 1$ and $p_S \perp z$, there is $\|p_S\|_2 + \|z\|_2 = 1$ and $\|p_S\|_2 \leq 1$. The objective becomes
    \[\bar x^\top \bar y -\bar s = \bar x^\top p_S + \bar x^\top z - \bar s.\]
    To maximize the objective, $z = t\cdot P_S\bar x$, where the matrix $P_S := I - A_S (A_S^\top A_S)^{-1} A_S^\top$ is the projection matrix to $\mbox{Null}(A_S^\top)$. Hence, provided $P_S \bar x \neq 0$, i.e, $\bar x \not\in \mbox{Span}(A_S)$, the solution
    \[y^* = p_S + \sqrt{1-\|p_S\|_2^2}\frac{P_S\bar x}{\|P_S\bar x\|_2}\]
    is global optimal if it satisfies \eqref{eq:cos_diff}.
    
    If $P_S \bar x = 0$, then any solution such that $y = p_S + z$ makes $\|y\|_2 = 1$ and satisfying \eqref{eq:cos_diff} is optimal. \qed
\end{proof}

\section{Parallel-Cut Management and Residual Feasibility} \label{app:too_parallel}

Outer-approximation implementations commonly reject nearly parallel cuts to limit redundancy and numerical ill-conditioning \citep{bienstock2026accurate,dai2026scheduling,dai2026solving}. More specifically, for two cuts $c_1^\top x\leq 0$ and $c_2^\top x\leq 0$, if the cosine of the angle formed by $c_1/\|c_1\|_2$ and $c_2/\|c_2\|_2$ is greater than $1-\epsilon$, then the method does not add the two cuts into the same model.

For two SOC supporting cuts $(y^{(i)})^\top x \leq \|y^{(i)}\|_2\cdot s$ and $(y^{(j)})^\top x \leq \|y^{(j)}\|_2\cdot s$, their coefficient vectors in the $(x,s)$-space are $(y^{(i)}/\|y^{(i)}\|_2, -1)$ and $(y^{(j)}/\|y^{(j)}\|_2, -1)$; hence, the cosine of the angle between the two cuts is
\[\frac{{y^{(i)}}^\top y^{(j)} + \|y^{(i)}\|_2\|y^{(j)}\|_2}{(\sqrt{2}\|y^{(i)}\|_2)(\sqrt{2}\|y^{(j)}\|_2)} = \frac{1}{2}(\frac{{y^{(i)}}^\top y^{(j)}}{\|y^{(i)}\|_2\|y^{(j)}\|_2} + 1).\]
Consequently, the two SOC cuts are classified as too parallel exactly when 
\[\frac{1}{2}(\frac{{y^{(i)}}^\top y^{(j)}}{\|y^{(i)}\|_2\|y^{(j)}\|_2} + 1) > 1-\epsilon;\]
\[\iff \frac{{y^{(i)}}^\top y^{(j)}}{\|y^{(i)}\|_2\|y^{(j)}\|_2} > 1 - 2\epsilon\]
Therefore, in our cut management, $\delta_{\max} = 1-2\epsilon$ provides a natural maximal value of the angular threshold in Problem~\eqref{eq:new_cut}. Increasing $\delta$ beyond $\delta_{\max}$ would no longer guarantee that a newly generated cut satisfies the prescribed parallel-cut criterion. In other words, once $\delta$ is beyond $\delta_{\max}$, we switch our cut management method to the classical parallel-cut management, i.e., reject the cut nearly parallel to existing cuts.

This cut-management rule may, however, leave a nonzero residual SOC violation. Consider an infeasible point $(\bar x, \bar s)\in P\backslash L^N$. Suppose that the deepest direction $\bar x/\|\bar x\|_2$ is classified as too parallel to some existing supporting direction $y^{(i)}/\|y^{(i)}\|_2$, i.e.,
\[\frac{\bar x^\top y^{(i)}}{\|\bar x\|_2\|y^{(i)}\|_2} \geq 1-2\epsilon.\]
Since $(\bar x, \bar s)\in P$, it satisfies the existing cuts. Consequently,
\begin{equation*}
    \begin{aligned}
        \|\bar x\|_2 - \bar s &\leq \|\bar x\|_2 - (y^{(i)}/\|y^{(i)}\|_2)^\top \bar x\\
        &= \|\bar x\|_2 (1-\frac{\bar x^\top y^{(i)}}{\|\bar x\|_2\|y^{(i)}\|_2})\\
        &\leq 2\epsilon\|\bar x\|_2.
    \end{aligned}
\end{equation*}
Thus, rejecting the deepest cut because of the parallelism guarantees a small SOC violation. But it does not necessarily guarantee a prescribed relative or absolute feasibility tolerance. If we set the relative feasibility tolerance as $\epsilon_{R}$, then, considering $\bar s < \|\bar x\|_2$ for an infeasible solution, the above bound guarantees
\[\frac{\|\bar x\|_2 - \bar s}{\max\{1, \|\bar x\|_2, \bar s\}} \leq \epsilon_R\]
whenever $2\epsilon \leq \epsilon_R$. However, the parallelism tolerance used in practice need not be sufficiently small to provide such a guarantee.

Furthermore, for some solvers like Gurobi \citep{gurobi} and BARON \citep{sahinidis1996baron}, there is also an absolute feasibility tolerance $\epsilon_A$. The above bound guarantees
\[\|\bar x\|_2 - \bar s \leq \epsilon_A\]
whenever $\|\bar x\|_2 \leq \frac{\epsilon_A}{2\epsilon}$. For larger $\|\bar x\|_2$, the bound alone does not ensure an absolute $\epsilon_A$-feasibility.

Hence, when Alg.~\ref{Alg:cut_generation} cannot obtain an additional admissible separating direction before reaching the parallelism threshold, the current OA solution may retain a small but non-negligible SOC violation. This illustrates the tension between avoiding nearly redundant cuts and recovering a sufficiently accurate primal feasible solution. To address this issue without sacrificing the proposed cut-management criterion, Section~\ref{sec:progressive-oa} supplements the outer approximation with the fixed-integer conic primal-recovery step.

\section{Algorithms}\label{app:whole_algorithm}

\subsection{OA algorithm}

The polyhedral outer approximation algorithm is summarized as Alg.~\ref{Alg:oa}.

\begin{algorithm}[!htbp]
	\SetAlgoLined
	\LinesNumbered
	\SetKwRepeat{Do}{do}{while}
    Set $\bar x := \arg\min_{x}\{c^\top x: x\in \mathcal R\}$\;
    \While{$\bar x \not\in \mathcal C \cap \mathcal X$}{
        Strengthen $\mathcal R$ by $\bar x$\;
        Set $\bar x := \arg\min_{x}\{c^\top x: x\in \mathcal R\}$\;
        \tcp{Conditions for an earlier termination}
    }
	\textbf{return} $\bar x$.
	\caption{Outer Approximation for MISOCP}
	\label{Alg:oa}
\end{algorithm}

\subsection{Cut generation algorithm with cut management}

The cut generation and cut management algorithm is summarized in Alg.~\ref{Alg:cut_generation}. 

In line 2, the algorithm begins with the deepest direction. If it satisfies all angular constraints, it is retained. Otherwise, the direction is rotated away from the closest conflicting existing cut according to \eqref{eq:rotate_cut} in line 4. 

In line 6, if the resulting direction satisfies the angular conditions but no longer separates $(\bar x, \bar s)$, $\delta$ is increased to relax the efficiency requirement. 

In line 10, if the selected value of $\delta$ does not guarantee a singleton active set, it is increased to the threshold of Theorem~\ref{thm:new_cut_rotating} and the construction is repeated. 

Finally, in line 12, if $\delta > 1 -2\epsilon$, we turn to the classical parallel-cut management. (See the discussion in Appendix~\ref{app:too_parallel}.)

\begin{algorithm}[!htbp]
	\SetAlgoLined
	\LinesNumbered
	\SetKwRepeat{Do}{do}{while}
	\KwIn{Initial $\delta_0$, infeasible solution $(\bar x, \bar s)$, $y^{(1)},\ldots,y^{(m)}$, and tolerance $\epsilon$.}
    Set $\delta := \delta_0$, $y^* = \bar x/\|\bar x\|_2$\;
    \While{$\exists i\in\{1,...,m\}, {y^*}^\top y^{(i)} > \delta$}{
        Set $i^* := \arg\max_{i}\frac{\bar x^\top y^{(i)}}{\|\bar x\|_2}$\;
        Set $y^* := \delta y^{(i^*)} + \sqrt{1-\delta^2}\frac{(I - y^{(i^*)}{y^{(i^*)}}^\top)y^*}{\|(I - y^{(i^*)}{y^{(i^*)}}^\top)y^*\|_2}$\;
        \eIf{${y^*}^\top y^{(i)} \leq \delta, \forall i\in\{1,...,m\}$}{
            \If{${y^*}^\top \bar x\leq \bar s$}{
                Set $\delta := \sqrt{\frac{\delta + 1}{2}}$, $y^* = \bar x/\|\bar x\|_2$\;}
        }{
            Set $\delta := \max_{i\neq j}\sqrt{\frac{{y^{(i)}}^\top y^{(j)}+ 1}{2}}$, $y^* = \bar x/\|\bar x\|_2$\;
        }
        \If{$\delta > 1-2\epsilon$}{
            \eIf{${y^*}^\top y^{(i)} \leq 1-2\epsilon,\ \forall i =1,...,m$}
            {\textbf{return} Cut $(y^*)^\top x \leq \|y^*\|_2\cdot s$, $\delta$.
            }{
                \textbf{return} NA\;
            }
        }
    }
	\textbf{return} Cut $(y^*)^\top x \leq \|y^*\|_2\cdot s$, $\delta$.
	\caption{Cut Generation Algorithm with Cut Management}
	\label{Alg:cut_generation}
\end{algorithm}

\subsection{Cut generation for one OA iteration}

In each iteration, if we obtain an infeasible solution, we generate cuts proposed in Section~\ref{sec:cutting_plane} for some violated cone constraints, as in Alg.~\ref{Alg:cut_generate_iteration}, and we call Alg.~\ref{Alg:cut_generation} as a subroutine for one second-order cone constraint.

\begin{algorithm}[!htbp]
	\SetAlgoLined
	\LinesNumbered
	\SetKwRepeat{Do}{do}{while}
	\KwIn{Parameters $\epsilon_\mathrm{tol}$, $p_\mathrm{cut}$, $\delta$, infeasible solution $\bar x$, relaxation $\mathcal P$, and the cone set $\mathcal C$.}
    Initialize 
    Check for $\epsilon_{\mathrm{tol}}$-violated cone constraints from $\mathcal C$\;
    Select $p_{\mathrm{cut}}$ most violated cone constraints\;
    For each selected constraint, call Alg.~\ref{Alg:cut_generation}. If Alg.~\ref{Alg:cut_generation} return cut and $\delta$, then update $\mathcal P$ and $\delta$\;
	\textbf{return} $\mathcal P, \delta$.
	\caption{Proposed Cut Generation for One Iteration}
	\label{Alg:cut_generate_iteration}
\end{algorithm}

In Alg.~\ref{Alg:cut_generate_iteration}, the parameter $\epsilon_\mathrm{tol}$ is the constraint feasibility tolerance. To control the size of generated cuts in each iteration, we only consider the $p_{\mathrm{cut}}$ most violated cone constraints in line 2. We note that this strategy has been tested in \citep{bienstock2026accurate,dai2026solving,dai2026scheduling}.

If we do not use the proposed cut and the cut management, then in line 2, we generate cuts defined in Proposition~\ref{prop:deepest_cut} and also reject the nearly parallel cuts with the same threshold $1-2\epsilon$.

\subsection{Progressive-Integrality outer-approximation method} \label{app:poa}

We modify Alg.~\ref{Alg:oa} by adding the LP stage, the IG stage, and the inner approximation trick proposed in Section~\ref{sec:progressive-oa} to get the progressive integrality outer approximation method for MISOCPs as Alg.~\ref{Alg:whole_algorithm}. We call Alg.~\ref{Alg:cut_generation} a subroutine in each iteration to strengthen the relaxation. In line 16, we check the relative constraint feasibility with $\epsilon_R = 2\epsilon$. \textbf{For the convergence of the algorithm, as discussed in Appendix~\ref{app:too_parallel}, with this setup, the cut generation method can find a feasible solution in the sense of relative constraint $\epsilon_R$-feasibility.}

\begin{algorithm}[!htbp]
	\SetAlgoLined
	\LinesNumbered
	\SetKwRepeat{Do}{do}{while}
	\KwIn{Original MISOCP, parameters $\epsilon, \epsilon_{\mathrm{LP}},\epsilon_{\mathrm{IG}},\epsilon_{\mathrm{MILP}},\epsilon_{\mathrm{converge}}, p_{\mathrm{frac}},\delta_0$, $\mathrm{max\_iter}$}
    Set $\delta = \delta_0$, $\mathrm{UB} := +\infty$, $\mathrm{LB} := -\infty$, $\mathrm{LB}_0 := -M$, $\mathcal P := \mathbb R^n$, and $\mathcal I:=\emptyset$\;
    \tcp{LP Stage}
    \Do{$\frac{|\mathrm{LB} - \mathrm{LB}_0|}{\max\{1, |\mathrm{LB}|, |\mathrm{LB}_0|\}}\geq \epsilon_{\mathrm{LP}}$}{
        Set $\bar x := \mathcal R_{\mathrm{LP}}$, $\mathrm{LB}_0:=\mathrm{LB}$, and $\mathrm{LB} := c^\top \bar x$\;
        Call Alg.~\ref{Alg:cut_generate_iteration} to update $\mathcal P$ and $\delta$\;
    }
    \tcp{IG Stage}
    \Do{$\frac{|\mathrm{LB} - \mathrm{LB}_0|}{\max\{1, |\mathrm{LB}|, |\mathrm{LB}_0|\}}\geq \epsilon_{\mathrm{IG}}$}{
        Select $p_{\mathrm{frac}}$ most fractional $\bar x_i$ over $\mathcal J$ and push their indexes to $\mathcal I$\;
        Set $\bar x := \mathcal R_{\mathrm{IG}}(\mathcal I)$, $\mathrm{LB}_0:=\mathrm{LB}$, and $\mathrm{LB} := c^\top \bar x$\;
        Call Alg.~\ref{Alg:cut_generate_iteration} to update $\mathcal P$ and $\delta$\;
        \If{$\mathcal I = \mathcal J$}{
            \textbf{Break}\;
        }
    }
    \tcp{Full MILPs}
    \For{$\mathrm{iter}\ = 1:\mathrm{max\_iter}$}{
        Set $\bar x := \mathcal R$, $\mathrm{LB}_0:=\mathrm{LB}$, and $\mathrm{LB} := c^\top \bar x$\;
        \tcp{Relative Feasibility Check\\ Guarantee to Find a Feasible Solution as the Discussion in Appendix D}
        \If{$\bar x\ \mathrm{is}\ 2\epsilon\mathrm{-feasibility}$}{
            \textbf{return} $\bar x$.
        }
        Call Alg.~\ref{Alg:cut_generate_iteration} to update $\mathcal P$ and $\delta$\;
        \tcp{Inner Approximation}
        \If{$\frac{|\mathrm{LB} - \mathrm{LB}_0|}{\max\{1, |\mathrm{LB}|, |\mathrm{LB}_0|\}} < \epsilon_{\mathrm{MILP}}$}{
            \If{$\bar{\mathcal{R}}(\bar x)\ \mathrm{is}\ \mathrm{feasbile} \land c^\top \hat x < \mathrm{UB}$}{
                Set $\hat x := \bar{\mathcal{R}}(\bar x)$, $\mathrm{UB}:=c^\top \hat x$\;
            }
        }
        \If{$\frac{|\mathrm{UB} - \mathrm{LB}|}{\max\{1, |\mathrm{UB}|, |\mathrm{LB}|\}} \leq \epsilon_{\mathrm{converge}}$}{
            \textbf{return} $\hat x$\;
        }
    }
    \textbf{return} $\hat x$.
	\caption{Progressive Integrality Outer Approximation for MISOCP}
	\label{Alg:whole_algorithm}
\end{algorithm}

\section{Why Progressive Integrality}\label{app:why_pioa}

\subsection{Mechanisms of bound improvement}
\label{subsec:poa-mechanisms}

The full mixed-integer relaxation $z_{\mathcal J}(\mathcal P)$ is the global dual bound of OA, and the LP relaxation $z_{\emptyset}(\mathcal P)$ is the root-node dual bound in branch-and-bound of the current full mixed-integer relaxation. The global dual bound determines the process of the whole OA, while the root-node dual bound impacts the size of the search tree and solution time of the branch-and-bound method for the current mixed-integer relaxation. Therefore, before explaining the motivation of the progressive integrality, we analyze the mechanisms of the improvement of dual bounds first.

\subsubsection{Transfer across integrality levels.} Suppose there are some globally valid cutting planes for outer-approximating $\mathcal C$, and let ${\mathcal H}$ be the intersection of halfspaces defined by these cutting planes. For two integrality set $\mathcal I \subseteq \mathcal K\subseteq \mathcal J$, define
\[\Delta_{\mathcal I}(\mathcal P, \mathcal H):=z_{\mathcal I}(\mathcal P\cap\mathcal H)-z_{\mathcal I}({\mathcal P}),\qquad
 G_{{\mathcal I},{\mathcal K}}({\mathcal P}):=z_{\mathcal K}({\mathcal P})-z_{\mathcal I}({\mathcal P}).\]
Here $\Delta_{\mathcal I}(\mathcal P, \mathcal H)$ measures improvement due to the new cuts alone, and $G_{{\mathcal I},{\mathcal K}}({\mathcal P})$ is the gap between two integrality levels without adding cuts.

\begin{proposition}\label{prop:poa-transfer}
For ${\mathcal I}\subseteq {\mathcal K}\subseteq {\mathcal J}$, given $\mathcal P$ and $\mathcal H$, we have
\begin{equation}\label{eq:poa-transfer}
 \Delta_{\mathcal K}(\mathcal P, \mathcal H)\ge \max\{\Delta_{\mathcal I}(\mathcal P, \mathcal H)-G_{{\mathcal I},{\mathcal K}}({\mathcal P}), 0\}.
\end{equation}
\end{proposition}
\begin{proof}
Increasing the integrality level gives $z_{\mathcal K}({\mathcal P}\cap\mathcal H)\ge z_{\mathcal I}({\mathcal P}\cap\mathcal H)$, and hence
\begin{equation*}
    \begin{aligned}
        z_{\mathcal K}({\mathcal P}\cap\mathcal H) - z_{\mathcal K}({\mathcal P}) &\geq z_{\mathcal I}({\mathcal P \cap \mathcal H})- z_{\mathcal K}(\mathcal P)\\
        &=z_{\mathcal I}({\mathcal P \cap \mathcal H})- z_{\mathcal K}(\mathcal P) + z_{\mathcal I}(\mathcal P) -z_{\mathcal I}(\mathcal P)\\
        &= \bigl(z_{\mathcal I}({\mathcal P \cap \mathcal H})-z_{\mathcal I}(\mathcal P)\bigr)-\bigl(z_{\mathcal K}(\mathcal P)-z_{\mathcal I}(\mathcal P)\bigr)\\
        &= \Delta_{\mathcal I}(\mathcal P, \mathcal H)-G_{{\mathcal I},{\mathcal K}}({\mathcal P}).
    \end{aligned}
\end{equation*}
Also, the left-hand side is nonnegative because $\mathcal P\cap\mathcal H\subseteq \mathcal P$.\qed
\end{proof}

Since $\mathcal J$ includes indices for all integer variables, $G_{\mathcal I,\mathcal J}(\mathcal P)$ is the integrality gap for the current integrality level $\mathcal I$, and $\Delta_\mathcal J(\mathcal P,\mathcal H)$ is the dual bound improvement for the full-MILP by the new added cuts alone.  From Equation~\eqref{eq:poa-transfer}, if $\Delta_{\mathcal I}(\mathcal P,\mathcal H) > G_{{\mathcal I},{\mathcal J}}({\mathcal P})$, then the excess improvement must transfer to the full-MILP dual bound. However, this is a sufficient threshold instead of a necessary condition: the full-MILP bound may increase even when the right-hand side of Equation~\eqref{eq:poa-transfer} is zero.

\subsubsection{Assignment decomposition for the cut improvement.}
Proposition~\ref{prop:poa-transfer} encourages a larger $\Delta_\mathcal I(\mathcal P, \mathcal H)$; thus, we discuss its relation with $\mathcal H$. For a fixed $\mathcal I$, let $\xi$ denote an assignment of $x$. Define
\[Q_\xi(\mathcal P,\mathcal I):=\{x\in \mathcal P\cap R(\mathcal X):x_i=\xi_i,\forall i\in \mathcal I\},\]
and $z_\xi(\mathcal P,\mathcal I) := \min_{x\in Q_\xi(\mathcal P,\mathcal I)}c^\top x$. Let $\Xi_{\mathcal I}$ denote the set of all feasible assignments of $x_i$ for $i\in\mathcal I$, and $z(\mathcal P,\mathcal I) := \min_{\xi \in \Xi_{\mathcal I}} z_\xi(\mathcal P,\mathcal I)$.
Define the gap between the current assignment and the optimal value as
\[g_\xi(\mathcal P) := z_\xi(\mathcal P,\mathcal I) - z(\mathcal P,\mathcal I),\]
and the change between the current relaxation $\mathcal P$ and the strengthened relaxation $\mathcal P\cap\mathcal H$ at the assignment $\xi$ as
\[\delta_{\xi,\mathcal I}(\mathcal P, \mathcal H) := z_\xi(\mathcal P\cap\mathcal H,\mathcal I) - z_\xi(\mathcal P,\mathcal I).\]

\begin{proposition}\label{prop:poa-replacement}
The cut-induced improvement satisfies
\begin{equation}
 \Delta_{\mathcal I}(\mathcal P, \mathcal H)=\min_{\xi\in \Xi_{\mathcal I}}\{g_\xi(\mathcal P)+\delta_{\xi,\mathcal I}(\mathcal P, \mathcal H)\}.
 \label{eq:poa-replacement}
\end{equation}
\end{proposition}

\begin{proof}
Let $Q_{\mathcal I}(\mathcal P) := \bigcap_{\xi\in\Xi_{\mathcal I}}Q_{\xi}(\mathcal P)$. We find that different assignments give disjoint components, e.g., for $\xi_1 \neq \xi_2$, 
\[Q_{\xi_1}(\mathcal P,\mathcal I) \cap Q_{\xi_2}(\mathcal P,\mathcal I) = \emptyset.\]
Hence, minimizing $c^\top x$ over $Q_{\mathcal I}$ is equivalent to first minimizing within each assignment and then choosing the best assignment, i.e.,
\begin{equation*}
    \begin{aligned}
        z_{\mathcal I}(\mathcal P) &= \min_{x\in Q_{\mathcal I}(\mathcal P)}c^\top x\\
        &= \min_{\xi\in\Xi(\mathcal P,\mathcal I)}\{\min_{x\in Q_\xi(\mathcal P)}c^\top x\}\\
        &=\min_{x\in \Xi(\mathcal P,\mathcal I)}z_\xi(\mathcal P).
    \end{aligned}
\end{equation*}
Similarly, there is $z_{\mathcal I}(\mathcal P\cap \mathcal H) = \min_{x\in \Xi(\mathcal P,\mathcal I)}z_\xi(\mathcal P \cap \mathcal H)$.

By definition,
\[\delta_{\xi,\mathcal I}(\mathcal P, \mathcal H) := z_\xi(\mathcal P\cap\mathcal H,\mathcal I) - z_\xi(\mathcal P,\mathcal I) \geq 0;\]
thus,
\[z_\xi(\mathcal P\cap\mathcal H,\mathcal I) = \delta_{\xi,\mathcal I}(\mathcal P, \mathcal H) + z_\xi(\mathcal P,\mathcal I).\]
Therefore,
\begin{equation}\label{eq:decomp_ph}
z_{\mathcal I}(\mathcal P\cap \mathcal H) = \min_{x\in \Xi(\mathcal P,\mathcal I)}z_\xi(\mathcal P \cap \mathcal H) = \min_{x\in \Xi(\mathcal P,\mathcal I)}\{z_\xi(\mathcal P) + \delta_{\xi,\mathcal I}(\mathcal H)\}.
\end{equation}

Recall that
\[g_\xi(\mathcal P) := z_\xi(\mathcal P,\mathcal I) - z(\mathcal P,\mathcal I),\]
so equivalently
\[z_\xi(\mathcal P,\mathcal I) = g_\xi(\mathcal P) + z(\mathcal P,\mathcal I).\]
Plugging into Equation~\eqref{eq:decomp_ph},
\begin{equation*}
    \begin{aligned}
        z_{\mathcal I}(\mathcal P\cap \mathcal H) &= \min_{x\in \Xi(\mathcal P,\mathcal I)}\{g_\xi(\mathcal P) + z(\mathcal P,\mathcal I) + \delta_{\xi,\mathcal I}(\mathcal P, \mathcal H)\}\\
        &= z(\mathcal P,\mathcal I) + \min_{x\in \Xi(\mathcal P,\mathcal I)}\{g_\xi(\mathcal P) + \delta_{\xi,\mathcal I}(\mathcal P, \mathcal H)\}.
    \end{aligned}
\end{equation*}
Therefore,
\[\Delta_{\mathcal I}(\mathcal P, \mathcal H) =z_{\mathcal I}(\mathcal P\cap \mathcal H) -  z(\mathcal P,\mathcal I)=\min_{\xi\in \Xi_{\mathcal I}}\{g_\xi (\mathcal P)+\delta_{\xi,\mathcal I}(\mathcal P, \mathcal H)\}.\]
\qed
\end{proof}

Proposition~\ref{prop:poa-replacement} implies a decomposition of the bound improvement from the new cuts. We further discuss the relation between $\delta_{\xi, \mathcal I}(\mathcal P, \mathcal H)$ and the cuts constructing $\mathcal H$. Let $Q := Q_{\xi}(\mathcal P,\mathcal I)$ be a nonempty polyhedron. Then 
\[z_{\xi}(\mathcal P, \mathcal I) = z(Q),\qquad z_{\xi}(\mathcal P\cap \mathcal H,\mathcal I) = z(Q\cap \mathcal P).\]  
Define the set $X^*(Q):=\arg\min_{x\in Q}c^\top x$, which is the optimal face of $Q$, and let
\[\rho(\mathcal H,Q):=\operatorname{dist}(Q\cap\mathcal H,X^*(Q))=\min_{x\in Q\cap\mathcal H,\,y\in X^*(Q)}\|x-y\|_2.\]
From the weak sharp minima by Burke and Ferries \citep{burke1993weak}, the solution set of an LP is weak sharp; thus, there exists $\alpha_{Q}>0$ such that
\begin{equation}
 c^\top x- z_\emptyset(Q)\ge\alpha_{\mathcal P}\cdot \operatorname{dist}(x,X^*(Q))
 \quad\text{for every }x\in Q.
 \label{eq:poa-growth}
\end{equation}
This property follows, for example, from classical Hoffman-type error bounds for linear systems (see e.g., \citep{abbasi2021strongly}) and can be interpreted as a measure of how quickly the objective increases as one moves away from the optimal face.

\begin{proposition}\label{prop:poa-sharpness} If the polyhedron $Q$ is nonempty, then
\begin{equation}
 \delta_{\xi,\mathcal I}(\mathcal P, \mathcal H) =z_\emptyset(Q\cap H)- z_\emptyset(Q)\ge\alpha_{Q}\cdot \rho(\mathcal H,Q).
 \label{eq:poa-displacement}
\end{equation}
\end{proposition}

\begin{proof}
We first show Equation~\eqref{eq:poa-growth} in our setup. If $X^*(Q)=Q$, \eqref{eq:poa-growth} holds for any positive $\alpha_{Q}$. Then we consider the case that $X^*(Q)\neq Q$. 

For any $y\in X^*(Q)$, consider the tangent cone of $Q$:
\[T_{Q}(y) := \{d:y+\epsilon d\in Q \mbox{ for some succiently small } \epsilon > 0\},\]
and the normal cone of $X^*(Q)$:
\[N_{X^*(Q)}(y) := \{w:w^\top(z - y)\leq 0, \forall z\in X^*(Q)\}.\]
Let
\[D_{Q}(X^*(Q)):=\{h:\exists y\in X^*(Q), h\in T_{Q}(y) \cap N_{X^*(Q)}(y), \|h\|_2 = 1\}.\]
We first show $D_{Q}(X^*(Q)) \neq\{0\}$. Take any $x\in Q \backslash X^*(Q)$ and let $y = \mbox{Proj}_F(x)$. Then because $x,y\in Q$ and $Q$ is convex,
\[x-y \in T_{Q}(y).\]
Also, by the projection optimality condition,
\[(x-y)^\top (z-y) \leq 0,\qquad \forall z\in X^*(Q);\]
thus, $x-y\in N_{X^*(Q)}(y)$. From $x\not\in X^*(Q)$, $x- y \neq 0$. Therefore,
\[x-y \in T_{Q}(y) \cap N_{X^*(Q)}(y)\backslash\{0\}.\]
Therefore, although for some $y$, the intersection may satisfy $T_{Q}(y) \cap N_{X^*(Q)}(y) = \{0\}$, we know that
\[\bigcup_{y\in X^*(Q)}\left(T_{Q}(y) \cap N_{X^*(Q)}(y)\right) \supset \{0\}.\]

For any $h\in D_{Q}(X^*(Q))$ and the solution $y\in X^*(Q)$ that $h\in T_{Q}(y) \cap N_{X^*(Q)}(y)$, there exists $\epsilon > 0$ such that $y + th \in Q$. Since $y$ is optimal,
\[c^\top (y+th)\geq c^\top y =z(Q),\]
we have $c^\top h \geq 0$. Suppose $c^\top h = 0$, then $y+th$ is optimal, i.e., $y+th\in X^*(Q)$. But $h\in N_{X^*(Q)}(y)$, and $h^\top(y-z)$ for all $z\in X^*(Q)$. Taking $z = y+th\in X^*(Q)$, we obtain,
\[h^\top[y+th - y] = t\|h\|_2^2 = t >0,\]
which is a contradiction. Therefore, $c^\top h > 0$.

Since $D_{Q}(X^*(Q)) \subseteq\{h:\|h\|_2 = 1\}$ is compact, we can define
\[\alpha_{Q} := \min_{h\in D_{Q}(X^*(Q))} c^\top h >0.\]

Again, take any $x\in Q \backslash X^*(Q)$ and let $y = \mbox{Proj}_{X^*(Q)}(x)$. Consider
\[h = \frac{x-y}{\|x-y\|_2}\in D_{Q}(X^*(Q)).\]
By the definition of $\alpha_{Q}$, $c^\top h \geq \alpha_{Q}$. Therefore, Equation~\eqref{eq:poa-growth} is given by
\[\begin{aligned}
    c^\top x - v &= c^\top(x- y)\\
    &\geq \|x-y\|_2\cdot c^\top(\frac{x-y}{\|x-y\|_2})\\
    &\geq \alpha_{Q}\cdot \|x-y\|_2\\
 (y = \mbox{Proj}_{X^*(Q)}(x))\quad   &=\alpha_{Q}\cdot\mbox{dist}(x, X^*(Q)).
\end{aligned}\]
Furthermore, minimizing the inequality over $x\in{Q \cap \mathcal H}$ gives Equation~\eqref{eq:poa-displacement},
\[\min_{x\in  {Q \cap \mathcal H}}\ c^\top x - z(Q) \geq \alpha_{Q}\cdot\min_{x\in  {Q \cap \mathcal H}}\mbox{dist}(x, F) = \alpha_{Q}\cdot \rho(\mathcal H,Q).\]\qed
\end{proof}

The improvement $\delta_{\xi, \mathcal I}(\mathcal P, \mathcal H)$ is determined by the constant $\alpha_Q$ and the distance from the new polyhedron $Q\cap \mathcal H$ to the old optimal face $X^*(Q)$. If the optimal face contains only one solution, then $\delta_{\xi, \mathcal I}(\mathcal P, \mathcal H)$ has a direct connection to the cut separation depth.

\begin{corollary}\label{cor:poa-sharpness}
    If $X^*(Q) = \{x^*\}$, $\mathcal H = \{a^\top x\leq b\}$, and $\|a\|_2 = 1$, let
    \[\eta = a^\top x^* - b > 0,\]
    and then
    \begin{equation}\label{eq:unique_opt}
        \delta_{\xi,\mathcal I}(\mathcal P, \mathcal H) = z_\emptyset(Q\cap\mathcal H) - z_\emptyset(Q) \geq \alpha_{Q}\cdot \eta.
    \end{equation}
\end{corollary}

\begin{proof}
Take any arbitrary $x\in {Q \cap \mathcal H}$. Since $x$ satisfies $a^\top x \leq b$,
\[a^\top(x^*-x) \geq a^\top x^* -b.\]
By the Cauchy-Schwarz inequality,
\[a^\top(x^*-x) \leq \|a\|_2\|x^*-x\|_2.\]
Therefore,
\[\|x^*-x\|_2 \geq a^\top x^* - b\geq \eta.\]
Since this holds for every $x\in {Q \cap \mathcal H}$,
\[\rho(\mathcal H, Q) = \mbox{dist}(Q\cap \mathcal H, X^*(Q)) = \mbox{dist}(Q\cap \mathcal H, x^*)\geq \eta.\]
Then Equation~\eqref{eq:unique_opt} is given by plugging $\rho(\mathcal H, Q)$ into Equation~\eqref{eq:poa-displacement}. 
\qed
\end{proof}

\subsubsection{Effect on the later MILP root gap.}
Define the root gap of the current full MIP relaxation:
\[\Gamma(\mathcal P) := z_{\mathcal J}(\mathcal P) - z_{\emptyset}(\mathcal P),\]
where $z_{\mathcal J}(\mathcal P)$ is the optimal value and $z_{\emptyset}(\mathcal P)$ is the objective value of the LP relaxation.
\begin{proposition}\label{prop:poa-rootgap}
    $\Gamma(\mathcal P\cap \mathcal H) = \Gamma(\mathcal P) - (\Delta_{\emptyset}(\mathcal P, \mathcal H) - \Delta_{\mathcal J}(\mathcal P, \mathcal H))$.
\end{proposition}

\begin{proof}
    By the definition of the root gap,
    \[\Gamma(\mathcal P\cap \mathcal H) = z_{\mathcal J}(\mathcal P \cap \mathcal H) - z_{\emptyset}(\mathcal P\cap \mathcal H).\]
    Plugging $\Delta_{\mathcal J}(\mathcal P, \mathcal H) = z_{\mathcal J}(\mathcal P\cap\mathcal H) - z_{\mathcal J}(\mathcal P)$ and $\Delta_{\emptyset}(\mathcal P, \mathcal H) = z_{\emptyset}(\mathcal P\cap\mathcal H) - z_{\emptyset}(\mathcal P)$,
    \[\begin{aligned}
        \Gamma(\mathcal P\cap \mathcal H) &= \Delta_{\mathcal J}(\mathcal P, \mathcal H) + z_{\mathcal J}(\mathcal P) - \Delta_{\emptyset}(\mathcal P, \mathcal H) - z_{\emptyset}(\mathcal P)\\
        &= z_{\mathcal J}(\mathcal P) - z_{\emptyset}(\mathcal P) - (\Delta_{\emptyset}(\mathcal P, \mathcal H) - \Delta_{\mathcal J}(\mathcal P, \mathcal H))\\
        &= \Gamma(\mathcal P) - (\Delta_{\emptyset}(\mathcal P, \mathcal H) - \Delta_{\mathcal J}(\mathcal P, \mathcal H)).
    \end{aligned}\]
    \qed
\end{proof}

If $\Delta_{\emptyset}(\mathcal P, \mathcal H) > \Delta_{\mathcal J}(\mathcal P, \mathcal H)$, $\mathcal H$ reduces the root gap; if $\Delta_{\emptyset}(\mathcal P, \mathcal H) = \Delta_{\mathcal J}(\mathcal P, \mathcal H)$, the root gap unchanges. Otherwise, the root gap increases. We note that it is possible that the root gap increases after adding some cuts to remove the best incumbent, since the primal bound can increase.

\subsection{Why use LP and IG stages?}
\label{subsec:poa-interpretations}

An immediate benefit from the progressive-integrality OA is that it reduces the cost of early OA iterations. LPs and partial MILPs generally require less computational effort than the corresponding full MILP. Beyond the runtime, the progressive scheme has more computational motivations as follows.

\subsubsection{Why start from the LP relaxation?} LP relaxation is not only computationally easier, but also brings benefits to the later MILPs. 

Suppose we generate cuts based on the LP relaxation and construct the intersection of halfspaces $\mathcal H$. From Proposition~\ref{prop:poa-replacement}, since $\mathcal I = \emptyset$, there is only one discrete component; thus,
\[\Delta_\emptyset(\mathcal P, \mathcal H) = \delta_{\xi,\emptyset}(\mathcal P, \mathcal H).\]
Together with Proposition~\ref{prop:poa-sharpness},
\[\Delta_\emptyset(\mathcal P, \mathcal H) \geq \alpha_{\mathcal P}\cdot \rho(\mathcal H, \mathcal P).\]
So if an LP cut substantially removes the old LP optimal face, it raises the root relaxation bound, i.e., $\Delta_\emptyset(\mathcal P, \mathcal H) > 0$.  

If $\Delta_\mathcal J(\mathcal P, \mathcal H) > 0$, then we improve the global dual bound of OA. If $\Delta_\mathcal J(\mathcal P, \mathcal H) = 0$, from Proposition~\ref{prop:poa-rootgap},
\[\Gamma(\mathcal P\cap\mathcal H) - \Gamma(\mathcal P) = -\Delta_{\emptyset}(\mathcal P, \mathcal H) + \Delta_{\mathcal J}(\mathcal P, \mathcal H) = -\Delta_{\emptyset}(\mathcal P, \mathcal H) < 0;\]
thus, the root gap decreases, thereby strengthening the branch-and-bound algorithm for the subsequent full MILPs.

However, suppose we consider the full MILP and generate cuts based on the optimal assignment $\xi^*$, i.e., $g_{\xi^*}(\mathcal P) = 0$. There is a potential near-optimial candidate $\xi^\prime \in \Xi(\mathcal P,\mathcal I)$ with a tiny $g_{\xi^\prime}(\mathcal P) + \delta_{\xi^*}(\mathcal P, \mathcal H)$. The first term $g_{\xi^\prime}(\mathcal P)$ is small because $\xi^\prime$ is near-optimal, and the second term $\delta_{\xi^\prime}(\mathcal P, \mathcal H)$ is small because the generated cuts do not consider this assignment. From Proposition~\ref{prop:poa-replacement},
\[\Delta_{J}(\mathcal P,\mathcal H) \leq g_{\xi^\prime}(\mathcal P) + \delta_{\xi^*}(\mathcal P, \mathcal H),\]
which can be tiny. We name this situation the \textbf{replacement effect}. In earlier stages of OA, the coarse relaxation contains many near-optimal integer configurations, and cuts generated from one configuration are localized while others are only weakly affected. For example, if there are multiple symmetrical optimal solutions, then cuts for one solution likely do not remove all symmetrical solutions; as a result, $\Delta_{J}(\mathcal P,\mathcal H) = 0$. Furthermore, if these generated cuts cannot move the old LP optimal face, then the LP relaxation improvement $\Delta_\emptyset(\mathcal P, \mathcal H) = 0$. From Proposition~\ref{prop:poa-rootgap},
\[\Gamma(\mathcal P\cap\mathcal H) - \Gamma(\mathcal P) = -\Delta_{\emptyset}(\mathcal P, \mathcal H) + \Delta_{\mathcal J}(\mathcal P, \mathcal H) = \Delta_{\mathcal J}(\mathcal P, \mathcal H) \geq 0;\]
thus, the root-gap potentially increases.

\subsubsection{Why progressively restore integrality?}
LP relaxation gives inexpensive and root-oriented strengthening, but its bound is limited by the continuous conic relaxation even if the conic approximation becomes exact; thus, it cannot in general replace later integrality enforcement. Also, from Proposition~\ref{prop:poa-transfer}, 
\[\Delta_{\mathcal J}(\mathcal H) \geq \max\{\Delta_{\mathcal I}(\mathcal P,\mathcal H) - G_{\mathcal I, \mathcal J}(P), 0\}.\]
Furthermore, for the fixed $\mathcal P$,
\[\mathcal I_1\subseteq \mathcal I_2\subseteq \mathcal J \implies G_{\mathcal I_2,J}(P)\leq G_{\mathcal I_1,J}(P).\]
Thus, imposing more integrality reduces the integrality gap. For an equal observed improvement $\Delta_\mathcal I(\mathcal P, \mathcal H)$, a smaller integrality gap gives a stronger transfer
guarantee. Therefore, when the improvement from LP relaxation is limited, moving to a partial-MILP level makes subsequent progress more directly relevant to the final MIP bound.

There is also a branch-and-bound interpretation. Consider an assignment $\xi$ fixing binary variables in $\mathcal I$, and the objective value to the corresponding partial MILP is
\begin{equation*} 
 z(\mathcal P, \mathcal I)=\min_{\xi\in\Xi_\mathcal I(\mathcal P)}z_\xi(\mathcal P, \mathcal I).
\end{equation*}
If a branch-and-bound tree branches first/early on variables from $\mathcal I$, the nodes at the resulting frontier correspond exactly to those assignments $\xi$. Hence the objective value $z(\mathcal P, \mathcal I)$ equals the dual bound of that truncated branch-and-bound tree. Similar to the analysis of the LP stage, cuts from the partial mixed-integer relaxations directly strengthen the dual bound of the upper levels of a subsequent branch-and-bound search, which helps decrease the size of the search tree and accelerate the solution \citep{land2009automatic}. Furthermore, the IG stage picks the most fractional variables to enlarge $\mathcal I$, which are also natural candidates in fractionality-based and hybrid branching strategies \citep{achterberg2009hybrid,achterberg2005branching}.

\subsubsection{Cut diversity.} Since the LP and partial-MILP relaxations have larger feasible domains than the full MILP relaxation, their optimal solutions may expose SOC infeasibility in a broader range of directions. Consequently, cuts generated across different integrality levels can potentially be more directionally diverse. In view of Section~\ref{sec:cutting_plane}, where nearly parallel supporting cuts were shown to have limited marginal efficiency, such diversity may further improve the resulting polyhedral approximation.

\section{Detailed Experiments} \label{app:experiment}

\subsection{Detailed Setup} \label{app:setup}

\paragraph{Data} Experiments are conducted on two groups of data: (1) CBLIB \citep{friberg2016cblib}; (2) the Central Illinois 200-bus test system and the South Carolina 500-bus test system \citep{birchfield2017}. The data and additional details for the power-system instances can be found in \cite[Appendix. E]{flores2022scheduling}.

\paragraph{Software and Hardware} All algorithms are implemented in Julia v1.11.5 \citep{bezanson2017julia} using JuMP v1.25.0 \citep{jump}. LP and MILP subproblems are solved by SCIP v8.0.2 \citep{scip1,scip2} or Gurobi v13.0.2 \citep{gurobi}, while continuous conic subproblems are solved by Mosek v11.0.9 \citep{mosek}. The experiments are conducted on a 64-bit Windows 11 desktop with an AMD Ryzen R5-5600G 3.90 GHz CPU (6 physical cores and 12 logical processors) and 32 GB RAM.

\paragraph{Configurations} For Alg.~\ref{Alg:cut_generation}, we set $\delta_0 = 0.99$, the parallel tolerance $\epsilon = (1e-5)/2$. For Alg.~\ref{Alg:cut_generate_iteration}, we follow \citep{bienstock2026accurate}, i.e., $\epsilon_{\mathrm{tol}} :=1e-5$ and $p_{\mathrm{cut}} :=55\%$. For Alg.~\ref{Alg:whole_algorithm}, we set $\epsilon_{\mathrm{LP}} = 1\%$, $\epsilon_{\mathrm{IG}} = 2\%$, $\epsilon_{\mathrm{MILP}} = 1\%$, $p_{\mathrm{frac}} = 10\%$. We set the maximum number of iterations as $\mathrm{max\_iter} := 25$. The target optimality gap $\epsilon_{\mathrm{convergence}}$ is specified separately for each experiment.

\paragraph{Performance Measurement} The optimality gap reported is
\[\mathrm{Gap} := \frac{|\mathrm{UB} - \mathrm{LB}|}{\max\{1, |\mathrm{UB}|, |\mathrm{LB}|\}}.\]
For methods A and B, the runtime reduction of A relative to B is reported as
\[\mathrm{Speedup} := \frac{\mathrm{Time}(B) - \mathrm{Time}(A)}{\mathrm{Time}(B)}.\]
For the mean value of a list $t = [t_1, ..., t_n]$, we adopt the geometric mean with a shift of 1 as
\[\mathrm{mean}\ \mathrm{value} := (\prod_{i=1}^n (t_i + 1))^{\frac{1}{n}} - 1.\]

\subsection{Detailed Illustrative Instances}\label{app:small}

The illustrative instances are sufficiently small that the time required for continuous SOCP solves and cut generation is negligible relative to the LP/MILP solution time; together, these operations require less than one second on average. Since the purpose of these experiments is to examine how progressive integrality and cut management affect the subsequent MILP solves, the runtime reported in Table~\ref{tab:small} and \ref{tab:small_detail} includes only LP and MILP solution times.

The columns \# LP, \# IG, and \# MIP denote the numbers of LP, partial-MILP, and full-MILP solves, respectively. MTime and MNodes measure the difficulty of the full MILPs encountered after the progressive stages. More precisely, if POA performs $n_1$ LP iterations and $n_2$ IG iterations before entering the full-MILP stage, MTime and MNodes are computed from the full MILPs solved after iteration $n_1 + n_2$. For the corresponding OA run, the same initial $n_1 + n_2$ iterations are excluded, so that the quantities compare MILPs encountered at comparable stages of the OA process.

We set the target optimality gap $\epsilon_{\mathrm{convergence}} = 1e-3$ in this test set.

\begin{table}[!htbp] 
    \centering 
    \small
    \setlength{\tabcolsep}{2.2pt}
    \begin{tabular}{l|l|c|ccc|cc|c} 
        \toprule[1pt] 
        Instance  &Method &Runtime &\# LP &\# IG &\# MIP &MTime &MNodes &Final Gap\\ 
        \midrule
        50\_0\_1\_w &OA &577 &0 &0 &8 &178 &212 &9.87e-4\\
        50\_0\_1\_w &OA-M &384 &0 &0 &7 &170 &129 &9.07e-4\\
        50\_0\_1\_w &POA &398 &4 &1 &3 &78 &27 &9.35e-4\\
        50\_0\_1\_w &POA-M &47 &4 &1 &1 &37 &1 &7.59e-4\\
        \midrule
        50\_0\_2\_w &OA &76 &0 &0 &6 &19 &1 &5.92e-4\\
        50\_0\_2\_w &OA-M &76 &0 &0 &6 &12 &1 &6.23e-4\\
        50\_0\_2\_w &POA &84 &3 &1 &2 &37 &1 &1.52e-4\\
        50\_0\_2\_w &POA-M &61 &3 &1 &1 &51 &1 &3.18e-4\\
        \midrule
        50\_0\_3\_w &OA &362 &0 &0 &7 &170 &298 &3.31e-4\\
        50\_0\_3\_w &OA-M &280 &0 &0 &7 &119 &17 &5.52e-4\\
        50\_0\_3\_w &POA &29 &4 &1 &1 &20 &1 &8.81e-4\\
        50\_0\_3\_w &POA-M &30 &4 &1 &1 &9 &1 &9.89e-4\\
        \midrule
        50\_0\_4\_w &OA &362 &0 &0 &7 &170 &291 &8.63e-4\\
        50\_0\_4\_w &OA-M &363 &0 &0 &7 &170 &27 &8.21e-4\\
        50\_0\_4\_w &POA &210 &4 &1 &2 &57 &27 &9.95e-4\\
        50\_0\_4\_w &POA-M &213 &4 &1 &2 &60 &29 &8.92e-4\\
        \midrule
        50\_0\_5\_w &OA &214 &0 &0 &6 &162 &419 &6.51e-4\\
        50\_0\_5\_w &OA-M &201 &0 &0 &6 &152 &61 &6.50e-4\\
        50\_0\_5\_w &POA &83 &4 &1 &1 &76 &1 &9.99e-4\\
        50\_0\_5\_w &POA-M &62 &4 &1 &1 &54 &1 &9.26e-4\\
        \bottomrule[1pt] 
    \end{tabular}
    \caption{Solution details for illustrative examples.}
    \label{tab:small_detail}
\end{table}

\subsection{Detailed Extended Instances}\label{app:medium}

For the extended CBLIB instances, the reported runtime is the total wall-clock solution time, including optimization subproblems and cut generation, since these experiments compare the proposed methods directly with Gurobi's native MISOCP solution. We use the target optimality gap $\epsilon_{\mathrm{convergence}} = 1e-4$.

For each instance, Runtime denotes the total time required to reach the target gap. The columns 1e-3 Gap and 5e-4 Gap report the first times at which the corresponding optimality gaps are reached. Best denotes the time at which the final best incumbent is first obtained, and Final Gap is the optimality gap at termination.

We report the detailed solutions case by case in Table~\ref{tab:medium_detail_gurobi}, \ref{tab:medium_detail_oa}, \ref{tab:medium_detail_oam}, \ref{tab:medium_detail_poa}, \ref{tab:medium_detail_poam}.

\begin{table}[!htbp] 
    \centering 
    \small
    \setlength{\tabcolsep}{2.2pt}
    \begin{tabular}{l|c|c|ccc|c} 
        \toprule[1pt] 
        Instance &Method &Runtime &1e-3 Gap &5e-4 Gap &Best &Final Gap\\ 
        \midrule
        100\_0\_1\_w &Gurobi &137 &23 &37 &74 &0.0088\%\\
        100\_0\_2\_w &Gurobi &376 &19 &50 &327 &0.0099\%\\
        100\_0\_3\_w &Gurobi &203 &21 &45 &203 &0.0085\%\\
        100\_0\_4\_w &Gurobi &235 &19 &45 &235 &0.01\%\\
        100\_0\_5\_w &Gurobi &183 &12 &39 &183 &0.01\%\\
        150\_0\_1\_w &Gurobi &337 &19 &85 &251 &0.0098\%\\
        150\_0\_2\_w &Gurobi &297 &19 &19 &219 &0.0099\%\\
        150\_0\_3\_w &Gurobi &435 &32 &61 &435 &0.0099\%\\
        150\_0\_4\_w &Gurobi &246 &20 &35 &226 &0.0097\%\\
        150\_0\_5\_w &Gurobi &187 &17 &50 &187 &0.0085\%\\
        200\_0\_1\_w &Gurobi &280 &32 &88 &280 &0.0093\%\\
        200\_0\_2\_w &Gurobi &426 &36 &81 &174 &0.0099\%\\
        200\_0\_3\_w &Gurobi &200 &31 &87 &198 &0.0084\%\\
        200\_0\_4\_w &Gurobi &315 &26 &78 &284 &0.01\%\\
        200\_0\_5\_w &Gurobi &300 &28 &51 &300 &0.0087\%\\
        200\_0\_6\_w &Gurobi &120 &28 &65 &94 &0.0099\%\\
        200\_0\_7\_w &Gurobi &162 &21 &68 &168 &0.0082\%\\
        200\_0\_8\_w &Gurobi &299 &33 &101 &299 &0.0095\%\\
        200\_0\_9\_w &Gurobi &152 &25 &25 &76 &0.0099\%\\
        200\_0\_10\_w &Gurobi &270 &39 &69 &270 &0.0085\%\\
        200\_0\_11\_w &Gurobi &320 &23 &82 &320 &0.0082\%\\
        200\_0\_12\_w &Gurobi &119 &45 &48 &72 &0.0099\%\\
        \bottomrule[1pt] 
    \end{tabular}
    \caption{Gurobi solution details on the extended CBLIB instances.}
    \label{tab:medium_detail_gurobi}
\end{table}

\begin{table}[!htbp] 
    \centering 
    \small
    \setlength{\tabcolsep}{2.2pt}
    \begin{tabular}{l|c|c|ccc|c} 
        \toprule[1pt] 
        Instance &Method &Runtime &1e-3 Gap &5e-4 Gap &Best &Final Gap\\ 
        \midrule
        100\_0\_1\_w &OA &137 &31 &31 &137 &0.0055\%\\
        100\_0\_2\_w &OA &367 &31 &171 &171 &0.0073\%\\
        100\_0\_3\_w &OA &302 &32 &132 &132 &0.0067\%\\
        100\_0\_4\_w &OA &404 &24 &151 &404 &0.0054\%\\
        100\_0\_5\_w &OA &123 &31 &31 &123 &0.0089\%\\
        150\_0\_1\_w &OA &577 &40 &168 &577 &0.0067\%\\
        150\_0\_2\_w &OA &289 &45 &123 &123 &0.0058\%\\
        150\_0\_3\_w &OA &497 &127 &303 &497 &0.0096\%\\
        150\_0\_4\_w &OA &235 &39 &39 &128 &0.0085\%\\
        150\_0\_5\_w &OA &183 &50 &50 &183 &0.0078\%\\
        200\_0\_1\_w &OA &247 &89 &89 &247 &0.0085\%\\
        200\_0\_2\_w &OA &829 &238 &238 &829 &0.0055\%\\
        200\_0\_3\_w &OA &447 &75 &75 &253 &0.0094\%\\
        200\_0\_4\_w &OA &352 &157 &157 &352 &0.0086\%\\
        200\_0\_5\_w &OA &395 &74 &74 &395 &0.0066\%\\
        200\_0\_6\_w &OA &290 &53 &53 &290 &0.0093\%\\
        200\_0\_7\_w &OA &281 &55 &162 &281 &0.0074\%\\
        200\_0\_8\_w &OA &387 &53 &152 &387 &0.0095\%\\
        200\_0\_9\_w &OA &434 &53 &53 &434 &0.0082\%\\
        200\_0\_10\_w &OA &347 &60 &60 &347 &0.0068\%\\
        200\_0\_11\_w &OA &277 &53 &133 &277 &0.0092\%\\
        200\_0\_12\_w &OA &179 &44 &44 &179 &0.0094\%\\
        \bottomrule[1pt] 
    \end{tabular}
    \caption{OA solution details on the extended CBLIB instances.}
    \label{tab:medium_detail_oa}
\end{table}

\begin{table}[!htbp] 
    \centering 
    \small
    \setlength{\tabcolsep}{2.2pt}
    \begin{tabular}{l|c|c|ccc|c} 
        \toprule[1pt] 
        Instance &Method &Runtime &1e-3 Gap &5e-4 Gap &Best &Final Gap\\ 
        \midrule
        100\_0\_1\_w &OA-M &138 &34 &34 &138 &0.0065\% \\
        100\_0\_2\_w &OA-M &297 &36 &126 &126 &0.0069\% \\
        100\_0\_3\_w &OA-M &201 &32 &105 &105 &0.0068\% \\
        100\_0\_4\_w &OA-M &318 &26 &164 &164 &0.0098\% \\
        100\_0\_5\_w &OA-M &331 &24 &100 &331 &0.0075\% \\
        150\_0\_1\_w &OA-M &404 &41 &218 &404 &0.0081\% \\
        150\_0\_2\_w &OA-M &258 &49 &135 &258 &0.0062\% \\
        150\_0\_3\_w &OA-M &256 &153 &153 &256 &0.0098\% \\
        150\_0\_4\_w &OA-M &233 &36 &125 &233 &0.0062\% \\
        150\_0\_5\_w &OA-M &226 &27 &27 &226 &0.0067\% \\
        200\_0\_1\_w &OA-M &225 &71 &225 &225 &0.0098\% \\
        200\_0\_2\_w &OA-M &648 &61 &239 &648 &0.0078\% \\
        200\_0\_3\_w &OA-M &459 &86 &86 &264 &0.0094\% \\
        200\_0\_4\_w &OA-M &343 &149 &149 &343 &0.0098\% \\
        200\_0\_5\_w &OA-M &455 &83 &83 &455 &0.0082\% \\
        200\_0\_6\_w &OA-M &302 &50 &97 &302 &0.0094\% \\
        200\_0\_7\_w &OA-M &157 &52 &52 &157 &0.0085\% \\
        200\_0\_8\_w &OA-M &341 &58 &166 &341 &0.0077\% \\
        200\_0\_9\_w &OA-M &277 &50 &147 &149 &0.0086\% \\
        200\_0\_10\_w &OA-M &159 &49 &49 &159 &0.0098\% \\
        200\_0\_11\_w &OA-M &322 &50 &185 &188 &0.0071\% \\
        200\_0\_12\_w &OA-M &169 &46 &46 &169 &0.0096\% \\
        \bottomrule[1pt] 
    \end{tabular}
    \caption{OA-M solution details on the extended CBLIB instances.}
    \label{tab:medium_detail_oam}
\end{table}

\begin{table}[!htbp] 
    \centering 
    \small
    \setlength{\tabcolsep}{2.2pt}
    \begin{tabular}{l|c|c|ccc|c} 
        \toprule[1pt] 
        Instance &Method &Runtime &1e-3 Gap &5e-4 Gap &Best &Final Gap\\ 
        \midrule
        100\_0\_1\_w &POA &101 &31 &31 &101 &0.0057\% \\
        100\_0\_2\_w &POA &176 &29 &29 &176 &0.0074\% \\
        100\_0\_3\_w &POA &182 &62 &62 &62 &0.0066\% \\
        100\_0\_4\_w &POA &229 &34 &34 &229 &0.0095\% \\
        100\_0\_5\_w &POA &135 &44 &44 &135 &0.0062\% \\
        150\_0\_1\_w &POA &250 &55 &55 &150 &0.0094\% \\
        150\_0\_2\_w &POA &245 &50 &50 &245 &0.0083\% \\
        150\_0\_3\_w &POA &240 &45 &45 &240 &0.0078\% \\
        150\_0\_4\_w &POA &153 &24 &48 &153 &0.0073\% \\
        150\_0\_5\_w &POA &235 &20 &20 &235 &0.0058\% \\
        200\_0\_1\_w &POA &232 &77 &77 &232 &0.0085\% \\
        200\_0\_2\_w &POA &260 &65 &65 &260 &0.0098\% \\
        200\_0\_3\_w &POA &252 &58 &58 &252 &0.0098\% \\
        200\_0\_4\_w &POA &268 &73 &73 &268 &0.0081\% \\
        200\_0\_5\_w &POA &262 &67 &67 &262 &0.008\% \\
        200\_0\_6\_w &POA &214 &67 &67 &214 &0.0062\% \\
        200\_0\_7\_w &POA &156 &60 &60 &156 &0.0058\% \\
        200\_0\_8\_w &POA &264 &70 &70 &264 &0.0088\% \\
        200\_0\_9\_w &POA &249 &54 &54 &249 &0.0072\% \\
        200\_0\_10\_w &POA &200 &44 &44 &200 &0.0063\% \\
        200\_0\_11\_w &POA &262 &101 &101 &262 &0.0059\% \\
        200\_0\_12\_w &POA &162 &46 &46 &162 &0.0063\% \\
        \bottomrule[1pt] 
    \end{tabular}
    \caption{POA solution details on the extended CBLIB instances.}
    \label{tab:medium_detail_poa}
\end{table}

\begin{table}[!htbp] 
    \centering 
    \small
    \setlength{\tabcolsep}{2.2pt}
    \begin{tabular}{l|c|c|ccc|c} 
        \toprule[1pt] 
        Instance &Method &Runtime &1e-3 Gap &5e-4 Gap &Best &Final Gap\\ 
        \midrule
        100\_0\_1\_w &POA-M &93 &27 &27 &93 &0.0067\% \\
        100\_0\_2\_w &POA-M &206 &61 &61 &61 &0.0088\% \\
        100\_0\_3\_w &POA-M &140 &34 &34 &140 &0.0094\% \\
        100\_0\_4\_w &POA-M &218 &33 &33 &218 &0.0095\% \\
        100\_0\_5\_w &POA-M &112 &35 &35 &35 &0.0084\% \\
        150\_0\_1\_w &POA-M &246 &51 &51 &246 &0.0093\% \\
        150\_0\_2\_w &POA-M &223 &44 &44 &223 &0.0074\% \\
        150\_0\_3\_w &POA-M &240 &45 &45 &240 &0.0079\% \\
        150\_0\_4\_w &POA-M &137 &20 &38 &137 &0.007\% \\
        150\_0\_5\_w &POA-M &204 &40 &40 &204 &0.0062\% \\
        200\_0\_1\_w &POA-M &207 &72 &72 &207 &0.0078\% \\
        200\_0\_2\_w &POA-M &253 &58 &58 &253 &0.007\% \\
        200\_0\_3\_w &POA-M &260 &33 &65 &260 &0.0055\% \\
        200\_0\_4\_w &POA-M &263 &68 &68 &68 &0.0077\% \\
        200\_0\_5\_w &POA-M &229 &56 &56 &229 &0.0074\% \\
        200\_0\_6\_w &POA-M &198 &62 &62 &62 &0.0066\% \\
        200\_0\_7\_w &POA-M &164 &60 &60 &164 &0.0051\% \\
        200\_0\_8\_w &POA-M &255 &60 &60 &60 &0.0091\% \\
        200\_0\_9\_w &POA-M &226 &52 &52 &52 &0.0057\% \\
        200\_0\_10\_w &POA-M &229 &62 &62 &62 &0.0057\% \\
        200\_0\_11\_w &POA-M &212 &47 &47 &47 &0.0066\% \\
        200\_0\_12\_w &POA-M &148 &38 &38 &148 &0.0066\% \\
        \bottomrule[1pt] 
    \end{tabular}
    \caption{POA-M solution details on the extended CBLIB instances.}
    \label{tab:medium_detail_poam}
\end{table}

\end{document}